%% file: main.tex
\documentclass{article}
\usepackage[utf8]{inputenc}
\include{preamble}

\title{Exponential ergodicity for a stochastic two-layer quasi--geostrophic model}
\author{%
  Giulia Carigi\footremember{math}{Department of Mathematics and Statistics, University of Reading, Reading, UK}\footremember{mpe}{Centre for the Mathematics of Planet Earth, University of Reading, Reading, UK}
  \and Jochen Br\"{o}cker\footrecall{math} \footrecall{mpe} \footnote{Department of Meteorology, University of Reading, Reading, UK} %
  \and Tobias Kuna\footrecall{math} \footrecall{mpe} \footnote{Dipartimento di Ingegneria e Scienze dell’Informazione e Matematica, Universit\`{a} degli Studi dell’Aquila, 67100 L’Aquila, Italy}%
  }

\begin{document}
\maketitle
\begin{abstract}
    Ergodic properties of a stochastic medium complexity model for atmosphere and ocean dynamics are analysed. More specifically, a two--layer quasi--geostrophic model for geophysical flows is studied, with the upper layer being perturbed by additive noise. This model is popular in the geosciences, for instance to study the effects of a stochastic wind forcing on the ocean. A rigorous mathematical analysis however meets with the challenge that in the model under study, the noise configuration is spatially degenerate as the stochastic forcing acts only on the top layer. 
    Exponential convergence of solutions laws to the invariant measure is established, implying a spectral gap of the associated Markov semigroup on a space of H\"older continuous functions.
    The approach provides a general framework for generalised coupling techniques suitable for applications to dissipative SPDE's.
    In case of the two--layer quasi--geostrophic model, the results require the second layer to obey a certain passivity condition.
\end{abstract}

\paragraph{Keywords:}generalized coupling; SPDEs; Stochastic geophysical flow models.
\paragraph{AMS Subject Classification:} \textit{Primary:} 37L40, 76U60 \textit{Secondary:} 60H15, 86A08 
%
\section{Introduction}
In this work we study the long time average behaviour of a stochastic version of an important model for large-scale atmosphere and ocean dynamics, the two--layer quasi--geostrophic (2LQG) model with a forcing on the top layer, to account, for example, for  the wind forcing on the upper ocean, composed of a deterministic and a stochastic part, which is white in time and coloured in space.
We will provide conditions for the model to be exponentially ergodic, which means that the transition probabilities converge to the unique stationary probability distribution with an exponential rate, and that temporal averages of an observable converge to averages of the observable with respect to that stationary probability distribution.
Indeed, showing exponential ergodicity for the stochastic 2LQG model and, in particular ergodicity, justifies that the long term average behaviour of the ocean dynamics at the mid--latitudes under stochastic wind stress can be studied using the unique probability distribution invariant under the dynamics. 
In general ergodicity is often a tacit assumption underlying applications, as it allows to draw conclusions on the statistics of the system from its time series.

In mathematical terms, for a system modeled by a stochastic partial differential equation (SPDE), as in this paper, exponential ergodicity means the following:
Consider the Markov semigroup and its transition probabilities $\{P_t, t\geq 0\}$ associated to it. When $\mu$ is an invariant measure with respect to the transition probabilities, we speak of exponential ergodicity if the transition probabilities $\{P_t, t\geq 0\}$ converge exponentially in time to an invariant measure $\mu$, independent of the initial data. This implies in particular that there exists exactly one invariant measure which has thus to be ergodic in the usual sense.

The quasi--geostrophic (QG) model, already present in the literature from the late 1930s, was systematically derived by Charney in 1948 in \cite{Charney48}. Used in early operational numerical weather forecasts, the QG model is still used extensively in research as it is strikes a balance between simplicity of formulation versus spectrum of the phenomena it can reproduce.
The QG model is an approximation of three dimensional Navier--Stokes equations on a rotating coordinate frame which best models the large-scale features (e.g.~1000 km for the atmosphere and 100 km for the ocean) of the atmosphere or the ocean at mid latitudes. 
%
Quasi--geostrophic models with several layers in particular are able to represent density stratification and provide insights into, for instance, atmosphere--ocean coupling and baroclinic instabilities. This type of instabilities is extremely common in both the atmosphere and ocean and is at the origin of large scale weather phenomena, for instance mid--latitude cyclones. The two--layer quasi--geostrophic (2LQG) model is one of the simplest models where the baroclinic instabilities arise.

%
%
The effect of a stochastic wind forcing on QG models has been a topic of research in meteorology and oceanography for at least thirty years, both in the single layer case e.g.~\cite{Griffa91, MOORE1999335, Treguier87} and the multi--layer case, see e.g.~\cite{berloff_2005, Kettyah2006,DelSole04,o2012emergence}.
In the mathematical literature the stochastic single--layer QG model, with either additive or multiplicative noise, has definitely received more attention (see e.g.~\cite{BRANNAN1998, Duan_ergodicityof, duan2004exponential, Desheng10}) than its multi--layer version. In fact, we can expect results achieved for the single layer to extend to the multi--layer case in situations where the random terms appear in all layers.
Less studied, though, is the action of a stochastic forcing acting only on one of the layers and  its consequent effects on the other layers and the whole dynamics. 

To the best of the authors' knowledge, the only reference for a mathematical investigation of a two--layer quasi--geostrophic model with a forcing on the top layer, prior to this work, is \cite{Chueshov01_Proba}. There the authors studied the long time dynamics of the model using the method of determining functionals for random dynamical systems. This method gives a way to parametrise the system's global attractor by means of a finite number of functionals. 
Furthermore, under some conditions on the parameters of the system, it is shown that functionals depending only on the top layer suffice to describe the attractor. However, this approach does not give information on the statistics of the model which will be the focus here.
With regards to exponential ergodicity, Harris' theorem provides conditions under which exponential convergence of transition probabilities holds in the total variation norm.
A major obstacle to applying this theorem in the context of SPDEs is that the transition probabilities may very easily be singular for different initial conditions.
In \cite[Section~4]{HMattSch11} a new framework is introduced to retrieve a version of Harris' theorem in the infinite dimensional context, which gives exponential rate of convergence in a Wasserstein semimetric, rather than in total variation.
The uniqueness and hence ergodicity of the invariant measure follow as well, and in most cases of interest also its existence.
However the assumptions required to apply this generalised Harris theorem may be challenging to show in practice. Recently \cite{butkovsky2020} provided a set of conditions which gives exponential and sub-exponential rate of convergence to the invariant measure in the Wasserstein semimetric, improving \cite{butkovsky2014} and \cite{KulikSch18}. 
%
The approach in \cite{butkovsky2020} relies on the \emph{generalised coupling} method.
The main idea of this method is to add a control to the stochastic forcing to estimate the distance between solutions with different initial data, or more precisely estimate the Wasserstein semimetric between their laws. 
%
%
This approach was introduced in the early 2000's (see e.g.~\cite{Matt2002, Hairer2001ExponentialMP, HMattSch11}) as \emph{asymptotic} coupling method, but in this case the added control ensures that the solutions with different initial data synchronise, namely have the same law asymptotically in time, at least with positive probability. The term \emph{generalised coupling} was introduced in \cite{KulikSch18}.


The results in \cite{butkovsky2020} are fairly general and apply to a wide variety of infinite dimensional systems. 
In the present paper, we present a methodology inspired by the results in \cite{butkovsky2020} but with a view on applications to stochastic dissipative SPDE's that describe atmosphere and ocean dynamics.
%
%
For such SPDE's we provide conditions sufficient for the results of \cite{butkovsky2020} (in particular the generalised Harris' theorem) to apply.
Our conditions are straightforward to check and easy to interpret mathematically and physically.
For the convenience of the reader, we provide proofs for those parts of \cite{butkovsky2020} where our line of reasoning deviates from~\cite{butkovsky2020}.

%

We will apply this technique to the stochastic two--layer quasi--geostrophic model to show exponential convergence of transition probabilities. Since the noise is acting only on one of the layers, the control will have to appear only in that layer but still be able to stabilise the lower layer. 
We will be able to find such a control by also imposing a condition involving the bottom friction to ensure the result. We may think of the imposed parameter condition as requiring the bottom layer (the one without noise) to be sufficiently dissipative so as to be determined by the top one. 
This provides information on the stability of the system for large times, in the sense that not only is there a unique invariant measure but the system approaches it exponentially fast in time, provided the bottom friction is large enough.

The results presented in this work are applicable (with appropriate modifications) to models with more than two layers. The condition on the passivity of the bottom layer would have to apply to the additional layers without stochastic forcing. 
Although the passivity condition might still be considered realistic in two layers with the second accounting for the lower ocean due to the friction with bottom of the ocean for example, instead for multiple layers a strong friction on all layers is harder to justify from a physical point of view, especially since the density difference between the layers (which would basically determine the magnitude of this friction) is not so sharp. How to remove this restriction will be subject to future research.

An important immediate consequence of the exponential convergence of the transition probabilities is the \emph{spectral gap} for the Markov semigroup on a suitable space of functions. In our case we clarify that these functions are locally H\"older continuous. 
The spectral gap is a crucial ingredient to prove linear response, which investigates how the invariant measure depends on the parameters of the system (see e.g.~\cite{HMajda10}). 
Linear response is a rigorous mathematical framework in which climate change triggered by changes in external parameters or forcings can be investigated (see e.g.~\cite{AbramovMadja2012,GhilLucarini20}). Studying linear response for the 2LQG model was one of the motivations behind establishing the spectral gap, and we will consider linear response in a forthcoming paper~\cite{response}.

\paragraph{Overview over the results} 
The stochastic two--layer quasi--geostrophic model, to be presented in \autoref{sec:themodel} in more detail, describes two layers of fluid on top of one another with certain heights and densities.
Both Coriolis effects and eddy viscosity is taken into account.
Forcing acts only on the top layer and has a non-trivial stochastic part which accounts for example for the effect of the wind on the upper ocean (with the second layer then representing the deeper ocean). 
To describe the model mathematically, let $\cD$ be a squared domain $\cD = [0,L]\times [0,L]\subset \R^2$ (e.g.~L is of the order $10^6$m for the atmosphere and $10^5$m for the ocean), and consider the equations 
\begin{align}
\label{eq:QG_stochasticIntro}
\begin{split}
    &d q_1 + J(\psi_1, q_1 + \beta y ) \, dt = \left(\nu\Delta^2\psi_1 \,+ f\right) dt + d W\\
    & \partial_t q_2 + J(\psi_2, q_2 + \beta y ) = \nu\Delta^2\psi_2 - r\Delta \psi_2,
\end{split}
\end{align}
 where $\x = (x, y)\in \cD$, $\bfpsi(t, \x) = (\psi_1(t, \x), \psi_2(t, \x))^t$ is the streamfunction of the fluid, and $\q(t,\x) = (q_1(t, \x), q_2(t, \x))^t $ is the so-called quasi--geostrophic potential vorticity. Vorticity and streamfunction are related through $\q = -\tA \bfpsi$ with an elliptic operator $\tA$ which also includes physical constants and parameters of the model (see Eq.~\eqref{eq:ch1defA}).
%
%
%
Moreover, $J$ is the Jacobian operator $J(a,b) = \nabla^{\perp}a \cdot \nabla b$, while $dW$ represents noise which is white in time but coloured in space, with trace~class covariance operator $Q$. Finally, the model includes a (time-independent) deterministic forcing on the top layer $f = f(\x)$.
Conditions ensuring the well--posedness of this model, along with further auxiliary results, are presented in Theorem~\ref{thm:solutions}.
%

In Section~\ref{sec:method} we consider a generic dissipative SPDE on a Hilbert space $(\cH, |\cdot|)$ with associated Markov semigroup $\cP_t$ which is Feller and satisfies \nameref{asuA} (see below).
Theorem~\ref{thm:exponentialergodicity} demonstrates that the transition probabilities are a strict contraction in a  Wasserstein semimetric associated with a suitable semimetric $\td$ on $(\cH, |\cdot|)$.
This implies, in particular, the existence of a spectral gap (see Corollary~\ref{cor:spectralgap}) as well as the generalised Harris' Theorem~\ref{thm:general_harris}.
The core \nameref{asuA} of Theorem~\ref{thm:exponentialergodicity} contains a~priori type energy estimates on the solution of the SPDE~(\nameref{ref:A2}), requirements on the control~(\nameref{ref:A1}, stabilisation of the dynamics, and \nameref{ref:A3}, regularity), and the existence of an appropriate Lyapunov function~(\nameref{ref:A4}).
In Section~\ref{sec:exp_stab_QG}, the results from Section~\ref{sec:method} are applied to the stochastic 2LQG model in Equation~\eqref{eq:QG_stochasticIntro}, giving us the following result (recall that $r$ determines the bottom friction, see 2nd line of Eq.~\ref{eq:QG_stochasticIntro}):
\renewcommand{\thethmmanual}{\ref{thm:exp_stab_QG}}%
\begin{thmmanual}[See Sec.~\ref{sec:exp_stab_QG} for precise statements]
There exists $r_0$ (depending on $\nu, Q$ and the forcing $f$ from Eq.~\ref{eq:QG_stochasticIntro}) so that if $r > r_0$, and $\range Q$ is large enough (depending on $r$ in a sense to be made precise),
        %
        then there exists a unique invariant measure $\mu_*$ as well as $t>0$ and $\rho <1$ such that 
        \begin{equation}
            \Wtd{P_t(\q_0, \cdot)}{P_t(\tq_0, \cdot)} \leq \rho\, \td(\q_0, \tq_0) 
        \end{equation}
        for all $\q_0, \tq_0\in \cH$.
        %
    \end{thmmanual}
The proof proceeds by showing that \nameref{asuA} is satisfied. 
In Remark~\ref{rmk:viscosity} it is demonstrated that Theorem~\ref{thm:exp_stab_QG} also holds for any given $r$ and $Q$, provided the viscosity $\nu$ is sufficiently large, see \eqref{eq:ch3condition_r_nu}.
The semimetric~$\td$ appearing in Theorem~\ref{thm:exponentialergodicity} is qualitatively of the form $\td(x, y) \cong |x - y|^\alpha $ for small $x, y \in \cH$, with $\alpha < \frac{1}{2}$.
Hence our spectral gap result Corollary~\ref{cor:spectralgap} refers to functions which are, roughly speaking, of H\"{o}lder type.
\paragraph{Acknowledgments}
The work presented here would have been impossible were it not for fruitful discussions with a number of colleagues. 
In particular, we are very grateful to Benedetta Ferrario, Franco Flandoli, Valerio Lucarini, and Jeroen Wouters for criticisms, comments, suggestions, and encouragement.
GC's work was funded by the Centre for Doctoral Training in Mathematics of Planet Earth, UK EPSRC funded (grant EP/L016613/1), by the LMS Early Career Fellowship (grant ECF1920-48), and by the UK EPSRC grant EP/W522375/1. 
Furthermore, GC would like to thank the Institute Henri Poincar\'{e} for supporting the participation to the thematic trimester on \textit{The Mathematics of Climate and the Environment} in the Autumn of 2019.
%
\section{The stochastic 2LQG model}
\label{sec:themodel}
The two--layer quasi--geostrophic model has been described as ``perhaps the most widely used set of equations for theoretical studies of atmosphere and ocean''~\cite{Vallis06}. 
From a mathematical point of view the 2LQG model is a system of several 2D~Navier--Stokes equations in vorticity formulation coupled to each other. In this section we lay down its precise mathematical formulation and the main notations following closely the set up described in \cite{Bernier94} and \cite{Chueshov01_Proba}. 
\subsection{Mathematical setup and notation}
The 2LQG equations are modelling two layers of fluid one on top of each other with mean height $h_1$ for the top layer and $h_2$ for the bottom one, and with density respectively $\rho_1$ and $\rho_2$ with $\rho_1< \rho_2$
%
We consider the so-called $\beta$-plane approximation to the Coriolis effect (see \cite[Section~2.3.2]{Vallis06}); this accounts for the fact that the vertical component of the rotation changes with the latitude $y$ by writing the Coriolis parameter as
\(
 f_c(y) = f_0 + \beta y,
\)
with $f_0$ and $\beta$ assigned positive constants.
%
We also take into account the effect of the eddy viscosity on both layers, and of the the bottom friction on the second layer.
We assume that the forcing acts only on the top layer and has a non-trivial stochastic part 
which accounts for example for the effect of the wind on the upper ocean. 

Let $\cD$ be a squared domain $\cD = [0,L]\times [0,L]\subset \R^2$. Consider the following equations 
\begin{align}
\label{eq:QG_stochastic}
\begin{split}
    &d q_1 + J(\psi_1, q_1 + \beta y ) \, dt = \left(\nu\Delta^2\psi_1 \,+ f\right) dt + d W\\
    & \partial_t q_2 + J(\psi_2, q_2 + \beta y ) = \nu\Delta^2\psi_2 - r\Delta \psi_2
\end{split}
\end{align}
 where $\x = (x, y)\in \cD$, $\bfpsi(t, \x) = (\psi_1(t, \x), \psi_2(t, \x))^t$ is the streamfunction of the fluid, and $\q(t,\x) = (q_1(t, \x), q_2(t, \x))^t $ is the so-called quasi--geostrophic potential vorticity. Vorticity and streamfunction are related through
 \begin{equation}
\label{eq:simple_relation_q_psi}
\begin{split}
     q_1 = \Delta \psi_1 + F_1(\psi_2 - \psi_1) \\
     q_2 = \Delta \psi_2 + F_2(\psi_1 - \psi_2)
\end{split}
\end{equation}
where $F_1, F_2$ are positive constants. 
Moreover, $J$ is the Jacobian operator $J(a,b) = \nabla^{\perp}a \cdot \nabla b$. $W$ is a random term, white in time and colored in space, more precisely, a so-called $Q$-Wiener process, which accounts for the stochastic part of the forcing (more details in Sec.~\ref{subsec:stoch_forcing}). Furthermore we assume periodic boundary conditions for $\bfpsi$ in both directions with period $L$ and we impose that 
\begin{equation}
\label{eq:ch1zeromeanvalue}
    \int_\cD \bfpsi(t, \x) \, d\x = 0 \fa t\geq 0.
    \end{equation}
The model includes a deterministic forcing on the top layer $f = f(\x)$ (time-independent) as well 
with zero spatial averages, i.e.
\begin{equation*}
    \int_\cD f(\x) \, d\x = 0 .
\end{equation*}
The constants $F_1, F_2$ are related to physical constants by
\begin{equation}
\label{eq:defF_i}
    F_i:= \frac{f_0^2}{g' h_i},
\end{equation}
with $g'$ the reduced gravity, $g' = g(\rho_2 - \rho_1)/ \rho_0$ where $\rho_0= (\rho_1 + \rho_2)/2$ is the characteristic value for the density, and $g$ is the gravitational acceleration. Further, denote by
\begin{equation}
\label{eq:ch1def_p}
    h_1 F_1 = h_2F_2 = \frac{f_0^2}{g'} =: p.
\end{equation} 
The model \eqref{eq:QG_stochastic} includes dissipation generated by the eddy viscosity on both layers modeled by the terms $\nu\Delta^2\psi_i$ and by the friction with the bottom modeled by $r\Delta \psi_2$.
We can write \eqref{eq:QG_stochastic} in vectorial formulation introducing 
\begin{equation}
    \label{eq: def B(U,V)}
  B(\bfpsi, \bfxi) = \left( \begin{array}{r}J( \psi_1, \Delta \xi_1) + F_1J(\psi_1, \xi_2) \\ J( \psi_2, \Delta \xi_2) + F_2J(\psi_2, \xi_1) \end{array} \right) .
\end{equation} 
Using the fact that $J(\psi,\psi) =0$, we can write more compactly \eqref{eq:QG_stochastic} as  
%
\begin{equation}
\label{eq:QG_stoc_vec}
    d \q + \left( B(\bfpsi, \bfpsi) + \beta \partial_x\bfpsi \right) \, dt= \nu \Delta^2\bfpsi \,dt+  \binom{f}{- r \Delta \psi_1} \, dt+ d\mathbf{W}
\end{equation}
where $\mathbf{W} = (W, 0)^t$, and $\Delta \bfpsi = (\Delta \psi_1, \Delta \psi_2)^t$. Moreover, we can express \eqref{eq:simple_relation_q_psi} vectorial as well
\begin{equation}
\label{eq:relation_q_psi}
    \q = (\Delta + M )\bfpsi \quad \text{with } 
    M = \begin{pmatrix} 
    - F_1 & F_1 \\F_2 & - F_2
    \end{pmatrix}.
\end{equation}

Next we introduce the notations for the mathematical setup we want to consider for the two--layer quasi--geostrophic.
Let $(L^2(\cD), \|\cdot\|_0)$, $(H^k(\cD), \| \cdot \|_k)$, $k\in \R$ be the standard Sobolev spaces of $L$-periodic functions satisfying \eqref{eq:ch1zeromeanvalue}. Denote by $(\cdot , \cdot)_k$ the associated scalar product. 
%
%
%
We also introduce appropriate norms on the product spaces to deal with our coupled system. Given $\bfpsi$ and $\bfxi$ elements of $H^k \times H^k$, $k>0$ or $L^2 \times L^2$ for $k = 0$, define 
\begin{align}
    (\bfpsi, \bfxi)_k := h_1(\psi_1, \xi_1)_{k} + h_2(\psi_2, \xi_2)_{k} \nonumber\\
    \|\bfpsi\|_k^2 := h_1\|\psi_1\|_{k}^2 + h_2\|\psi_2\|_{k}^2.  \label{Hk_h_norm}
\end{align}
Then we define 
\begin{align*}
   \mathbf{L}^2 &= \left\lbrace \bfpsi\in L^2 \times L^2 \, : \, \|\bfpsi\|_0^2 < \infty \right\rbrace \quad \\
   \bH^k &= \left\lbrace \bfpsi\in H^k \times H^k \, : \, \|\bfpsi\|_k^2 < \infty \right\rbrace, \quad k > 0 
\end{align*}
and we denote with $\bH^{-k}$ the dual space of $\bH^k$, $k>0$. 

The Poincar\'{e} inequality in $\bH^k$ reads as 
\begin{equation}
\label{eq:Poincare}
    \| \bfpsi \|_k \leq \lambda_1^{-1/2}\| \bfpsi \|_{k+1}, 
\end{equation}
where $\lambda_1$ is the smallest eigenvalue of the operator $-\Delta$. 

Define the operator $\tA: \mathbf{H}^{k+2} \to \mathbf{H}^{k}$, $k\in \R$, connection of the streamfunction with the quasi--geostrophic potential vorticity
\begin{equation}
 \label{eq:ch1defA}
    \tA\bfpsi = - (\Delta + M)\bfpsi, \quad \bfpsi \in \mathbf{H}^{k+2}.
\end{equation}
It is easy to see that $\tA$ is an unbounded non--negative self--adjoint operator in $\bH^k$ with respect to the weighted scalar product $(\cdot , \cdot)_k$, indeed 
\begin{equation}
\label{eq:quadtA}
    (\tA\bfpsi, \bfpsi)_k = -(\Delta \bfpsi, \bfpsi)_k - (M\bfpsi, \bfpsi)_k
    = \| \bfpsi \|^2_{k+1} + p |\psi_1 - \psi_2|_k^2,
\end{equation}
which, due to \eqref{eq:ch1zeromeanvalue}, has an inverse which is bounded as function $\mathbf{H}^{k} \to \mathbf{H}^{k+2}$. Then for each $\q\in \bH^k$ there exists $\bfpsi\in \mathbf{H}^{k+2}$ such that $\q = - \tA \bfpsi$.
It is interesting to note that if the upper layer is more dense than the bottom, namely $\rho_1\geq \rho_2$, then $F_1, F_2$ would change sign and $\tA =-( \Delta - M)$. Therefore $\tA$ would not be non--negative and one could not develop a consistent theory for such equations, reflecting the physical impossibility of this setup.
%
%


\begin{remark}
    Since the $\bLtwo$ and $\bH^1$ norms and the $\bLtwo$ scalar product are the most used throughout this work, for the sake of simplifying notation we denote them as follows
    \begin{align*}
        |\psi| := \| \psi\|_0 &\quand \| \psi \| := \|\psi\|_1 \\
        |\bfpsi| = h_1|\psi_1| + h_2|\psi_2| := \| \bfpsi\|_0 &\quand \| \bfpsi \| = h_1 \| \psi_1\| + h_2 \|\psi_2\| := \| \bfpsi\|_1\\
        (\psi,\xi) := (\psi,\xi)_0 &\quand (\bfpsi, \bfxi) = h_1(\psi_1, \xi_1) + h_2(\psi_2, \xi_2)
    \end{align*}
\end{remark}

Finally, we introduce two new norms on the level of the potential vorticities. 
For $\q\in \Hmo$ there exists $\bfpsi\in \bH^1$ such that $\q = -\tA\bfpsi$, and we can define the norm on $\Hmo$
\begin{equation*}
    \vertiii{\q}_{-1}^2 := \| \bfpsi\|^2 + p | \psi_1 - \psi_2|^2
\end{equation*}
and, for $\q\in \bLtwo$ with $\bfpsi \in \bH^2$ define the norm on $\bLtwo$
\begin{equation*}
    \vertiii{\q}_0^2 :=|\Delta \bfpsi|^2 + p \|\psi_1 - \psi_2\|^2.
\end{equation*}
%
%
Note that by Poincar\'{e} inequality one has
\begin{align*}
\begin{split}
   \vertiii{\q(t)}^2_{-1} &= \| \bfpsi\|^2 + p |\psi_1 - \psi_2|^2 \\
   &\leq \lambda_1^{-1}\left( |\Delta \bfpsi|^2 + p \| \psi_1 - \psi_2\|^2\right) = \lambda_1^{-1} \vertiii{\q(t)}_0^2.
  \end{split}
\end{align*}
Furthermore, these norms are equivalent to $\| \cdot \|_{-1}$ and $\|\cdot \|_{0}$ respectively and have a series of useful properties, which we will show now and exploit later on. 

\begin{lemma}
\label{lemma:ch1propertiesnorms}
    Consider $\q\in \Hmo$ and $\bfpsi \in \mathbf{H}^1$ such that $\q = -\tA\bfpsi $. Then the following relations hold:
    \begin{align}
        -( \q, \bfpsi) = \vertiii{\q}_{-1}^2 \label{eq:ch1(u,v)=|u|*} \\
        \|\bfpsi \|^2 \leq \vertiii{\q}_{-1}^2 \leq a_0 \|\bfpsi \|^2 \label{eq:ch1*normpoincare}
    \end{align}
    for some $a_0> 0$.
    For $\q\in \bLtwo$ and $\bfpsi \in \mathbf{H}^2$ such that $\q = -\tA\bfpsi $, we have:
    \begin{align}
        ( \q, \Delta \bfpsi) =  \vertiii{\q}_0^2  \label{eq:ch1(u,v)=|u|2}\\
      |\Delta \bfpsi|^2 = |\Delta \bfpsi|^2\leq \vertiii{\q}_{0}^2 \leq a_0 |\Delta \bfpsi|^2 . \label{eq:ch12normpoincare}
    \end{align}
\end{lemma}

\begin{proof}
    Equality \eqref{eq:ch1(u,v)=|u|*} is a direct calculation, cf. also \eqref{eq:quadtA}.
   In a similar way we can show \eqref{eq:ch1(u,v)=|u|2}. By definition of $\q$ we have
     \begin{equation*}
       ( \q, \Delta\bfpsi) =  (\Delta\bfpsi, \Delta\bfpsi) + (M\bfpsi, \Delta \bfpsi) .
     \end{equation*}
     Then by definition of $M$ \eqref{eq:relation_q_psi}, relation \eqref{eq:ch1def_p} and Green's theorem 
     \begin{align*}
     \begin{split}
         (M\bfpsi, \Delta \bfpsi)& = -h_1F_1(\psi_1 - \psi_2, \Delta \psi_1) + h_2F_2(\psi_1 - \psi_2, \Delta \psi_2) \\
       &=  p ((-\Delta)^{1/2} (\psi_1 - \psi_2), (-\Delta)^{1/2}\psi_1) - p ((-\Delta)^{1/2} (\psi_1 - \psi_2), (-\Delta)^{1/2}\psi_2). 
     \end{split}
     \end{align*}
     Therefore we have 
     \begin{equation*}
         ( \q, \Delta\bfpsi) =  |\Delta\bfpsi|^2 + p \| \psi_1 - \psi_2\|^2 =  \vertiii{\q}_0^2. 
     \end{equation*}
    
    Moving on to \eqref{eq:ch1*normpoincare}, the lower bound follows from \eqref{eq:quadtA} and the upper bound is a consequence of the Poincar\'{e} inequality \eqref{eq:Poincare} and the parallelogram law. Indeed we have 
    \begin{align*}
        p |\psi_1 - \psi_2 |^2 &\leq  p \lambda_1^{-1}\| \psi_1 - \psi_2 \|^2 \\
        &\leq  \frac{2p}{\lambda_1\min(h_1, h_2)} ( h_1\| \psi_1 \|^2 + h_2\| \psi_2 \|^2)  \\
         &= 2\lambda_1^{-1}\max(F_1, F_2) \| \bfpsi\|^2.
    \end{align*}
    Setting $a_0$ to be $1 +2\lambda_1^{-1}\max(F_1, F_2) $, we see that \eqref{eq:ch1(u,v)=|u|2} holds.
    
    With similar arguments 
    \eqref{eq:ch12normpoincare} can be shown as well.
   
\end{proof}

\autoref{tab:norms} contains a summary of the spaces and relative norms used throughout this work. 
\begin{table}[ht]
    \centering
    \begin{tabular}{ll}
    Space & Norm \\[1ex] 
     $\bH^{k}$ & $\|\bfpsi\|_k^2 = h_1\|\psi_1\|_k^2 + h_2\| \psi_2\|_k^2$ \\ 
     $\bLtwo$ = $\bH^0$ & $|\bfpsi|^2= h_1|\psi_1|^2 + h_2|\psi_2|^2$ \\ 
     $\bH^1 $& $\|\psi\|^2 = h_1\|\bfpsi_1\|^2 + h_2\| \psi_2\| ^2$ \\[1ex]
      $\bLtwo$
      & $\vertiii{\q}_0^2 = \| \bfpsi\|_2^2 + p \| \psi_1 - \psi_2\|^2 $\\ 
      $\Hmo$
      & $\vertiii{\q}_{-1}^2 = \| \bfpsi \|^2 + p | \psi_1 - \psi_2|^2$
      
    \end{tabular} 
    \caption{\label{tab:norms}Notations for the two--layer quasi--geostrophic model. 
            Rows~1--3 will be mainly used for the streamfunctions $\bfpsi$. 
            Rows~4,5 will be mainly used for the potential vorticities $\q$, i.e.\ functions in the range of the operator $-\tA$. 
            Note that $\q = -\tA \bfpsi$ in the second block.}
\end{table}

    
    %
     Finally, standard bounds on  the Jacobian see for example \cite[Lemma~3.1]{Chueshov01_Proba}, yield the following bound for the bilinearity $B$:
    \begin{lemma}[\cite{thesis}]
    \label{lemma:ch1propertiesB}
    Let $B$ be the bilinear operator defined in \eqref{eq: def B(U,V)}, then for  $\bfpsi,\bfxi, \bfphi\in \bH^2$ 
    \begin{align}
        (B(\bfpsi,\bfxi), \bfphi) = - (B(\bfphi,\bfxi), \bfpsi), \label{eq: (B(U,V), W)= - (B(W,V), U) } \\
       (B(\bfpsi,\bfxi), \bfpsi) = 0. \label{eq: (B(U,V), U) = 0} 
    \end{align}
    Moreover, for $\bfpsi, \bfxi \in \bH^2$, there exists positive constant $k_0$ such that
    \begin{equation}
        |(B(\bfpsi,\bfpsi),\bfxi)|\leq k_0 \| \bfpsi\| |\Delta \bfpsi|  |\Delta \bfxi| .
        \label{eq:bound(B(u,u),v)}
    \end{equation}
    \end{lemma}

    \subsection{The stochastic forcing} 
    \label{subsec:stoch_forcing}
    On the first layer we consider a forcing with a stochastic component which is white in time and colored in space. More precisely, consider the probability space $(\Omega, \mathcal{F}, \mathbb{P})$ and let $W$ be a $L^2$--valued Wiener process on it with covariance operator $Q$. We assume $Q$ to be trace class in $L^2$, namely given $\lbrace e_k \rbrace_{k\in \N}$, a complete orthonormal basis of $L^2$ 
    \begin{equation*}
        \Tr Q := \sum_{k=1}^{\infty} \langle Q e_k, e_k \rangle  < \infty.
    \end{equation*} 
    We also assume that $Q$ and $-\Delta$ commute, meaning that the $\{e_k\}_{k \in \N}$ can be taken as the eigenvectors of $-\Delta$.
Note that all the $e_k$ have to fulfill \eqref{eq:ch1zeromeanvalue} and hence $\int W(t)(\x) d\x =0$ as well, but this follows from the fact that we defined the spaces $L^2$ and $H^k$ so that \eqref{eq:ch1zeromeanvalue} holds.
For later use we note the following facts. 
Let $\cH$ be a general separable Hilbert space and $W$ a Wiener process with values in $\cH$ and trace--class covariance operator $Q$.
    It can be shown that $W(t)$ can be written as a sum of real valued Wiener processes, namely
    \begin{equation*}
        W(t)= \sum_{k=1}^{\infty} \sqrt{\sigma_k}\beta_k(t)e_k,
    \end{equation*}
    where $\sigma_k$ are eigenvalues of $Q$ with $\lbrace e_k\rbrace_{k\in \N}$ a corresponding  orthonormal system of eigenfunctions, and $\beta_k(t)$, $k\in \N$, are independent real valued Brownian motions on $(\Omega, \mathcal{F}, \mathbb{P})$.
    Then, we define for each $n\in \N$ the orthogonal projection $\Pi_n: \cH \to \Span\{ e_1, \ldots,e_n\}  $. Note that (see e.g.~\cite[Section~4.2.2]{DaPratoZab2014}) $\Pi_nW$ and $(1 - \Pi_n)W$ are independent Wiener processes with covariance matrices 
    \begin{equation}
        \label{eq:covarianceQn}
        Q_n = \sum_{k= 0}^n \sigma_k e_k\otimes e_k 
                \quad  \text{resp.} \quad 
        Q^n = \sum_{k= n+1}^\infty \sigma_k e_k\otimes e_k. 
    \end{equation}
%
%
 Furthermore, let $\tW_n$ be a $n$-dimensional Brownian motion, equal in law to $\Pi_n W $. Then, as all $\beta_k$ are mutually independent, $\tW(t):=\tW_n + (1 - \Pi_n)W$ is a $Q$-Wiener process equal in law to $W$.
\subsection{Solutions of the stochastic 2LQG model}
\label{subsec:solutions} 
The deterministic version of \eqref{eq:QG_stoc_vec} has been shown to be well--posed in \cite{Bernier94}, while, for the stochastic model, the following result holds:
    %
\begin{theorem}[\cite{thesis}]
\label{thm:solutions}
        Consider the system \eqref{eq:QG_stoc_vec} with initial condition $\q_0\in \Hmo$, deterministic forcing $f\in H^{-2}$ and the covariance operator $Q$ trace class in $L^2$. Then there exists a pathwise unique weak solution, namely for a.a. $\omega \in \Omega$ there exists a unique $\q$ such that 
         \begin{equation*}
            \q\in C([0,T]; \Hmo)\cap L^2(0,T; \bLtwo) 
        \end{equation*}
        and that satisfies \eqref{eq:QG_stoc_vec} in the integral sense: 
        \begin{multline*}
             \left( \q(t), \varphi \right) + \int_0^t \left( B(\bfpsi, \bfpsi), \varphi \right) + \beta \left( \partial_x \bfpsi, \varphi \right)\, ds = (\q_0, \varphi) + \nu \int_0^t \left( \Delta\bfpsi, \Delta \varphi \right)  \;ds  \\ 
             + \int_0^t h_1( f, \varphi_1 )- rh_2(\Delta \psi_2, \varphi_2) \;ds +  h_1\left( W(t), \varphi_1 \right)
        \end{multline*}
        for all $\varphi= (\varphi_1, \varphi_2)^t \in \bH^2$, $t \in [0,T]$.
        Furthermore, $\q(t, \q_0)$ is continuous with respect to the initial condition $\q_0$ as a function in $\Hmo$.
    \end{theorem}
    The well--posedness of Equation~\eqref{eq:QG_stoc_vec} was shown in \cite{thesis}, using the same approach as in \cite{Flandoli94} for Navier--Stokes. The main idea is to reformulate Equation~\eqref{eq:QG_stoc_vec} into an equation for $\tq: = \q - (\eta, \, 0)^t$, where $\eta$ is an auxiliary Ornstein-Uhlenbeck process, 
    namely, a solution $\eta$ of 
    \begin{equation*}
        d\eta - \alpha \Delta \eta \, dt = dW
    \end{equation*}
    for an appropriate value of the constant $\alpha>0$.
    Since the covariance matrix $Q$ is trace class in $L^2$, by the classic theory of SDEs we know $\eta$ has a continuous version with values in $L^2$. 
    As $\tq$ now satisfies a deterministic equation with random coefficients, existence and uniqueness can be established with PDE~techniques as done for the multi--layer~QG model in \cite{Bernier94}.
     In \cite{thesis} it is shown that there exists a unique $\tq\in C([0,T]; \Hmo)\cap L^2(0,T; \bLtwo)$ so that $\tq + \eta$ is solution of \eqref{eq:QG_stoc_vec}. This ensures the well-posedness of \eqref{eq:QG_stoc_vec} as given in \autoref{thm:solutions}. 
    Furthermore the solution can be shown to be continuous with respect to the driving noise, meaning that for any $t\geq 0$ and initial condition $\q_0$ there exists a continuous function 
    \[ \Phi_t^{0} : C([0,t]; L^2) \to \Hmo \]
    such that $\q(t, \q_0) = \Phi_t^0\left(\{W(s)\}_{s\leq t}\right)$. This follows from the fact that $\tq$ depends continuously on $\eta$ and that the Ornstein--Uhlenbeck process can be expressed as continuous function of the driving noise. 
    This reformulation of \eqref{eq:QG_stoc_vec} was first introduced in \cite{Chueshov01_Proba}, where the associated random dynamical system is studied and the existence of an absorbing set is demonstrated. 
    %
    %
    %
    
    %
    On the Hilbert space $\Hmo$ we consider the Borel $\sigma$--algebra $\cB(\Hmo)$.  We denote by $B_b(\Hmo)$ the space of real bounded Borel measurable functions on $\Hmo$, by $C_b(\Hmo)$ the continuous bounded functions, and by $\cM_1(\Hmo)$ the space of probability measures on $\Hmo$. Given the uniqueness of the solutions by \autoref{thm:solutions}, it can be shown that $\{\q(t, \q_0), \, t \geq 0\}$ is a Markov process (see e.g.~Theorem~9.14 in \cite{DaPratoZab2014}), where $\q(t, \q_0)$ is the solution seen as a random variables, with the randomness introduced by the driving Brownian motion. For any $\q_0\in \Hmo$ and $A\in \cB(\Hmo)$, and $t>0$, define the associated Markov transition probabilities as 
    \begin{equation*}
        P_t(\q_0, A) = \bP \left( \q(t, \q_0)\in A \right) = \Law (\q(t, \q_0))(A).
    \end{equation*}
    The corresponding Markov semigroup is defined as
    \begin{equation*}
        \cP_t\varphi(\q_0) = \int \varphi(x) P_t(\q_0, dx) = \E\, \varphi(\q(t, \cdot ; \q_0))
    \end{equation*}
    for any $\varphi\in B_b(\Hmo)$, and $\cP_t^*$ is its dual acting on $\cM_1$ i.e.
    \begin{equation*}
        \cP_t^*\mu(A) = \int \cP_t \mathbbm{1}_A(x) \, \mu(dx) =  \int P_t(x, A) \, \mu(dx).
    \end{equation*}
    
    Given the continuous dependence of $\q$ on the initial condition it follows that the semigroup $\cP_t$ is \textit{Feller} namely that for any $\varphi \in C_b(\Hmo)$ and any $t\geq 0$ one has $\cP_t \varphi \in C_b(\Hmo)$.
%
\section{Methodology}
\label{sec:method}
We discuss a generalised Harris theorem for Markov processes in a Hilbert space.
The presentation closely follows~\cite[Section~4]{HMattSch11} and, building on a framework developed in \cite{butkovsky2020}, we provide conditions particularly suitable for application to dissipative SPDEs.
\subsection{Basic definitions and notation}
Let $\cH$ be a Hilbert space with scalar product $(\cdot, \, \cdot)$ and associated norm $|\cdot|$ and Borel $\sigma$--algebra $\cB(\cH)$.
We denote by $B_b(\cH)$ the space of all real bounded Borel functions and by $\cM_1(\cH)$ the space of probability measures on $\cH$.
Consider two probability measures $\mu_1, \mu_2 \in \cM_1(\cH)$.
A probability measure $\Gamma$ on $\cH \times \cH$ is called \emph{coupling} of $\mu_1, \mu_2$ if its marginals agree with $\mu_1, \mu_2$; more specifically, if we let $\pi_1(x,y) = x$, $\pi_2(x,y) = y$ be the projections of $\cH \times \cH$ onto its two components, then
\begin{equation*}
    \pi_1^* \Gamma  = \mu_1, \quand \pi_2^* \Gamma = \mu_2.
\end{equation*}
We denote the set of all such couplings as $\mathcal{C}(\mu_1, \mu_2)$. Equivalently we call a pair of random variables $(\xi_1, \xi_2)$ a coupling of $\mu_1, \mu_2$ if $\Law \xi_1 = \mu_1$ and $\Law \xi_2 = \mu_2$.
%
%

%
A function $d: \cH\times \cH \to \R_+$ is a {\em semimetric} (sometimes also referred to as a {\em distance--like function}) when it is symmetric, lower semi-continuous and such that $d(x,y) = 0 \Leftrightarrow x = y $. When the symmetry fails, we refer to $d$ as a \textit{premetric}.  
 %
%
A semimetric $d$ on $\cH$ can be lifted to a semimetric on the level of probabilities called \textit{Wasserstein semimetric} $ W_{d}$: given two probability measures $\mu_1, \mu_2\in \mathcal{M}_1(\cH)$ set
    \begin{equation*}
        W_{d}(\mu_1, \mu_2) := \inf_{\Gamma \in  \mathcal{C}(\mu_1, \mu_2)} \int d(x, y) \; \Gamma(dx, dy).
    \end{equation*}
%
%
The classic coupling lemma (see e.g.~\cite[Theorem~4.1]{Villani}) ensures that the infimum in this definition is always reached by a coupling, given the lower semi-continuity of $d$.
Further, note that, if $d$ satisfies the triangular inequality, namely it is a metric, then $W_d$ is a metric on $\cM_1$.
An important example is the {\em total variation distance} $d_{TV}(\mu_1, \mu_2)$ between two probability measures $\mu_1, \mu_2 \in \mathcal{M}_1$, defined as the Wasserstein metric associated with the discrete metric $d(x,y) = \mathbbm{1}(x\neq y)$.
%
%
%
There is an important connection between Wasserstein semimetrics and weak norms associated with Lipschitz functions.
Given a semimetric $d$, define the associated Lipschitz seminorm as 
     \begin{equation}
         \label{eq: lipschitz metric}
        \|\varphi\|_d := \sup_{x\neq y }\frac{|\varphi(x) - \varphi(y)|}{d(x, y)}.
     \end{equation}
     It is immediate that $ \|.\|_{d}$ is positive homogenous and satisfies the triangle inequality. 
     However, $\| \varphi\|_{d}= 0$ only implies that $\varphi$ is a constant function.
For any two probability measures $\mu_1, \mu_2 \in \mathcal{M}_1$ it is easy to see that
    \begin{equation}
    \label{eq:kantorovich}
       \sup_{\|\varphi\|_{d}\leq 1} \left| \langle \varphi, \mu_1 \rangle -  \langle \varphi, \mu_2 \rangle\right| \leq  W_{d}(\mu_1, \mu_2).
    \end{equation}
  When $d$ is a metric however, the Kantorovich-Rubinstein formula (see e.g.~\cite{dudley_2002}) states that~\eqref{eq:kantorovich} is an equality.
%
%
\subsection{A generalisation of Harris' theorem}        
In a finite dimensional context Harris' theorem (see e.g.~\cite[Theorem~1.5]{HMattSch11}, \cite{YetAnother}, \cite{meyn_tweedie}) provides conditions under which the transition probabilities of a Markov process converge to an invariant measure in the total variation distance. 
   In fact it ensures convergence if there exists a so-called \emph{small set} which is visited infinitely often by the process, and the speed of convergence is related to how fast the process returns to such sets. A set $A \subset \cH$ is \emph{small} if there exists a time $t>0$ and a constant $\varepsilon>0$ such that 
    \begin{equation*}
        d_{TV} \left( P_t(x, \cdot) , P_t(y, \cdot)\right) \leq 1 - \varepsilon \fa x, y \in A.
    \end{equation*}
%
%
As explained for example in \cite{YetAnother}, typical candidates for small sets are the level set of so called Lyapunov functions.
        A measurable function $V: \cH \to \R_+$  is called \emph{Lyapunov function} for $\cP_t$ if there exist positive constants $C_V$, $\gamma$, $K_V$ such that 
        \begin{equation*}
        \label{def:LyapunovfctHMS}
            \cP_t V(x) \leq C_V e^{-\gamma t} V(x) + K_V \fa x\in \cH, \, t\geq 0. 
        \end{equation*}
%
%
However, a major difficulty with applying Harris' theorem in an infinite dimensional context is that the transition probabilities $P_t(x, \cdot )$ and $P_t(y, \cdot)$ might be singular for different initial conditions $x \neq y$. In that case one has that
     \(d_{TV} \left( P_t(x, \cdot) , P_t(y, \cdot)\right)  = 1\)
and there are no small sets.
In \cite{HMattSch11} in order to retrieve a version of Harris' theorem in infinite dimensions, a weaker notion of small set is introduced, one which uses a Wasserstein semimetric, rather than the total variation distance: 
   \begin{definition}[{\cite[Definition~4.4]{HMattSch11}}]
   \label{def:dsmallset}
        Let $\cP$ be a Markov operator on $\cH$ with associated transition function $P(\cdot, \cdot)$ and let $d: \cH \times \cH \to [0,1]$ be a semimetric on $\cH$. A set $K\subset \cH$ is called \emph{d-small} for $\cP$ if there exists $\varepsilon>0$ such that 
        \begin{equation*}
            W_d(P(x, \cdot), P(y, \cdot)) \leq 1 - \varepsilon
        \end{equation*}
        for all $x,y\in K$.
   \end{definition}
   Note that the range of the semimetric is restricted to the unit interval.
   This does not, however, impose a restriction,  since if $d$ is a semimetric, then so is $d \wedge 1$, and furthermore $W_{d \wedge 1} \leq W_d$.
   The reason why the range of $d$ is restricted in this way is that large values of $x$ and $y$ will be dealt with separately using the Lyapunov function, as we will see.
    There is a last ingredient necessary to give a general form of Harris' theorem for Wasserstein semimetrics: the semimetric has to be contracting for the semigroup $\lbrace \cP_t, \, t\geq 0 \rbrace$. 
    In the original statement of Harris' theorem this condition is not included as it is automatically verified by the total variation distance. 
   \begin{definition}[{\cite[Definition~4.6]{HMattSch11}}]
   \label{def:contracting}
        Let $\cP$ be a Markov operator on $\cH$ with associated transition function $P(\cdot, \cdot)$. Then a semimetric $d: \cH \times \cH \to [0,1]$ is called \emph{contracting} for $\cP$ if there exists $\alpha < 1$ such that for every pair $x,y \in \cH$ with $d(x,y)< 1$ 
        \begin{equation*}
            W_d(P(x, \cdot), P(y, \cdot)) \leq \alpha\, d(x, y).
        \end{equation*}
   \end{definition}
    Going through the proof of Theorem~4.8 in \cite{HMattSch11}, it is clear that first and foremost the following intermediate result is shown:
    \begin{theorem}[Generalised Harris' theorem, {\cite[Theorem~4.8]{HMattSch11}}]
    \label{thm:general_harris}
        Let $\cP_t$, $t\geq 0$, be a Markov semigroup over $\cH$ admitting a continuous Lyapunov function $V$. Suppose that there exists $T>0$ and a semimetric $d: \cH \times \cH \to [0,1]$ which is contracting for $\cP_{T}$, and such that the level set $\lbrace x\in \cH \, : \, V(x) \leq 4K_V \rbrace$ is $d$-small for $\cP_{T}$. Then the following holds:
        \begin{enumerate}
            \item Defining $\td(x,y)^2 = d(x,y)(1 + V(x) + V(y))$, there exists $t_*>0$ and $\rho < 1$ such that
        \begin{equation}
        \label{eq:ch4general_harris}
            \Wtd{P_{t_*}(x, \cdot)}{P_{t_*}(y, \cdot)} \leq \rho\, \td(x,y).
        \end{equation}
        \item The semigroup $\lbrace \cP_t, \, t\geq 0 \rbrace$ has at most one invariant measure $\mu_*$. 
       \item If there exists a complete metric $d_0$ on $\cH$ such that $\cP_t$ is Feller in $(\cH, d_0)$ and $d_0 \leq \sqrt{d}$, then the semigroup $\lbrace \cP_t, \, t\geq 0 \rbrace$ has an invariant measure.
        \end{enumerate}
    \end{theorem}
    From this result two important corollaries will follow, namely the exponential convergence of transition probabilities to the invariant measure, \autoref{cor:ch4expconvergence}, and a spectral gap result, \autoref{cor:spectralgap}. 
    \begin{corollary}
    \label{cor:ch4expconvergence}
        Suppose the assumptions of \autoref{thm:general_harris} hold uniformly for $T$ belonging to an open interval of $\mathbbm{R}$. Then \autoref{thm:general_harris} implies that there exist $\gamma>0$ and a $t_0>0$  such that given $\mu, \nu \in \mathcal{M}_1$
        \begin{equation*}
            W_{\td}(\cP_{t} \mu, \cP_{t}\nu) \leq e^{-\gamma t}  W_{\td}(\mu, \nu)\fa x\in \cH, \,  t\geq t_0.
        \end{equation*}
        Furthermore, if $\mu_*$ is the invariant measure for $\cP_t$, there exists $C>0$ such that 
        \begin{equation*}
            W_{\td}(P_t(x, \cdot), \mu_*) \leq C(1 + V(x)) e^{-\gamma t} \fa x\in \cH,\, t\geq t_0. 
        \end{equation*}
    \end{corollary}
 Recall that $\mu_*$ being invariant for $\cP_t$ means that $\mu_*$ is an eigenvector of $\cP_t^*$ with eigenvalue~1. Then the next corollary implies that the remaining spectrum is contained in the disk of radius $e^{-\gamma t}$ around the origin, or we may say $\cP^*_t$ exhibits a spectral gap. 
    \begin{corollary}[Spectral gap]
    \label{cor:spectralgap}
        Suppose all conditions in \autoref{thm:general_harris} are met and let $\| \cdot \|_{\td}$ be the Lipschitz seminorm \eqref{eq: lipschitz metric} associated to the semimetric $\td$. Then there exists $\rho<1$ and $t_*> 0$ such that 
        \begin{equation*}
            \| \cP_{t_*} \varphi - \langle \varphi , \mu_*\rangle \|_{\td} \leq  \rho \| \varphi - \langle \varphi , \mu_*\rangle\|_{\td}
        \end{equation*}
        for all $\varphi:\cH\to \R$ such that $\| \varphi\|_{\td}< \infty$. 
    \end{corollary}
Given the formulation of $\td$, for observables such that $\| \varphi\|_{\td}< \infty$ we have
\begin{equation*}
   |\varphi(x) - \varphi(y)| \leq  \| \varphi\|_{\td} \sqrt{d(x,y)(1 + V(x)+ V(y))} \fa x\neq y.
\end{equation*}
This implies that the spectral gap holds for functions that are locally H\"{o}lder with exponent $\frac{1}{2}$ with respect to the semimetric $d$.
Yet the precise regularity of those observables with respect to the original norm $|\cdot|$ on $\cH$ depends on the formulation of the contracting semimetric $d$.
Recently \cite{butkovsky2020} provided a set of conditions to construct a semimetric satisfying the hypothesis in \autoref{thm:general_harris}. 
Given this choice of the semimetric, it turns out that the spectral gap holds for functions that are locally H\"{o}lder with respect to the original norm $|\cdot|$ and some exponent $\alpha < \frac{1}{2}$. 
In the next subsection we propose a modification of those conditions in \cite{butkovsky2020} to facilitate their application to models like the stochastic 2LQG.
%
%
\subsection{Framework for SPDEs}
\label{sec:exp_stab_general}
The following setup includes the dissipative SPDE's we are interested in but potentially other interesting dynamic models in infinite dimensions.
 Let $(\cH, |\cdot|)$ and $(\cV, \| \cdot \|)$ be Hilbert spaces with $\cV \subset \subset \cH$  (i.e.\ $\cV$ is compactly contained in $\cH$ or the unit sphere of $\cV$ is relatively compact in $\cH$).
 Further, $\|v\| \geq |v|$, and $\cV$ is assumed dense in $\cH$.
 This implies that $\cH = \cH' \subset \subset \cV'$, that $|v| \geq \|v\|'$, and that $\cH$ is dense in $\cV'$.
 Consider the stochastic equation 
  \begin{equation}
  \label{eq:ch4generalSDE}
      dX = \left( AX + F(X) \right) \, dt + dW_X, \quad X(0) = x
  \end{equation}
  where $A:\cV \to \cV'$ is a nonnegative linear operator, $W_X$ is a Wiener process on $(\Omega, \mathcal{F}, \bP)$ with values in $\cH$ and trace class covariance operator $Q:\cH \to \cH$, and $F:\cV \to \cV'$ is a nonlinear continuous function such that Equation~\eqref{eq:ch4generalSDE} holds in $\cV'$, and that there exists a unique solution for any initial condition $X(0) = x\in \cH$. As for the two--layer quasi--geostrophic equations, we assume the solution to be in $C([ 0,T]; \cH)\cap L^2(0,T; \cV)$ for all $T >0$ and a.a. $\omega\in \Omega$, and to be continuous with respect to the initial condition (as a function into $C([ 0,T]; \cH)$). 
  We define the Markov semigroup $\cP_t$ on $\cB_b(\cH)$ by
  \( \cP_t \varphi (x) = \E \, \varphi (X(t; x))\)
  and denote the associated transition probabilities as $P_t(x, \cdot)$. Given the regularity of the solutions the associated semigroup is Feller.

To ensure the generalised Harris' theorem applies we need to find a semimetric $d$ which is contracting with respect to $\cP_t$ for some $t$ and show that there is a Lyapunov function $V$ with a $d$-small sublevel set.
By definition then we have to appropriately bound the distance $W_d(P_t(x, \cdot), P_t(y, \cdot))$, for $x,y$ such that $d(x,y)<1$ in order to get the contraction, and for $x,y$ in a sublevel set of $V$ in order to get the $d$--smallness.

Let $Y$ satisfy the stochastic equation
  \begin{equation}
  \label{eq:ch4generalSDEII}
      dY = \left( AY + F(Y) \right) \, dt + dW_Y, \quad Y(0) = y
  \end{equation}
with $A, F$ as in Equation~\eqref{eq:ch4generalSDE}, while $W_Y$ and $W_X$ have the same law but are not necessarily identical (note also the different initial condition $Y(0) = y \neq x$).
As a consequence $\Law Y(t) = P_t(y, \cdot)$, so any coupling between $W_Y$ and $W_X$ furnishes a coupling between $X(t)$ and $Y(t)$ (or equivalently between $P_t(x, \cdot)$ and $P_t(y, \cdot)$) implying that $\E(d(X(t), Y(t)))$ is an upper bound for $W_d(P_t(x, \cdot), P_t(y, \cdot))$.
Taking the discrete metric $d(x,y)= \mathbbm{1}(x\neq y)$ for instance, we find $1 - \bP(X(t) = Y(t))$ as an upper bound for the total variation metric.
The event $X(t) = Y(t)$ though implies, roughly speaking, that the process $X$ has forgotten its initial condition.
Finding a coupling however that guarantees this with nontrivial probability is difficult in our infinite--dimensional setting.
While memory of the initial condition will eventually get wiped out in those degrees of freedom directly affected by the noise, the required time might get very long if there are infinitely many such degrees of freedom.
Clearly, this memory might also survive in degrees of freedom where there is no noise at all.

The idea of {\em generalised coupling} 
is to consider an intermediate equation 
    \begin{equation}
    \label{eq:ch4equationYtilde}
        d\tY = \left( A\tY + F(\tY) + G(X, \tY)\right) \, dt + d W_X, \quad \tY(0) = y
    \end{equation}
    where $A, F$ and even $W_X$ are as in Equation~\eqref{eq:ch4generalSDE}, but 
 with initial condition $y$ as in Equation~\eqref{eq:ch4generalSDEII}. 
 Further, $G$ is a control function determined in such a way that on the one hand for a $t$ large enough we can estimate $\E(d(X(t), \tY(t)))$ even if $x \neq y$ (using an appropriate semimetric $d$).
 On the other hand $G$ has to be chosen so that, under appropriate conditions, the law of $\tY$ will be absolutely continuous with respect to the law of $Y$ and so that the total--variation metric between the laws of $\tY$ and $Y$ can be controlled (as in \autoref{lemma:dTVBM}).
 We will achieve this by choosing $G$ such that the law of the process 
\begin{equation}\label{eq:ch4tildeWn}
    \tW_X(t) = \int_0^t G(X, \tY) \, ds + W_X(t)
\end{equation}
 is absolutely continuous with respect to the law of $W_Y$.
Let us now collect the assumptions we will be using and which are particularly suitable for applications to dissipative SPDEs. 
    %
    \paragraph{Assumption~A.}\label{asuA}
     Let $\cH_n$ be an $n$-dimensional subspace of $\cH$ and denote by $\Pi_n$ the orthogonal projection on $\cH_n$.
    The covariance operator $Q$ commutes with $\Pi_n$, and furthermore $Q_n := \Pi_n Q$ is invertible on $\cH_n$.
     Given a solution $t \to X(t)$ of \eqref{eq:ch4generalSDE}, there exists a finite dimensional measurable control $G: \cH\times \cH \to \cH_n$ such that the controlled equation \eqref{eq:ch4equationYtilde} has a unique solution $\tY$ in the sense described at the beginning of this subsection. 
    %
    In addition, we require the following:
    \begin{itemize}
        \item[{\crtcrossreflabel{\textbf{A1}}[ref:A1]}] There exist $\kappa_0>0$ and $\kappa_1\geq 0$ such that for all $t\geq 0$
        \begin{equation*}
            |X(t) - \tY(t)|^2\leq |x - y|^2 \exp(- \kappa_0 t + \kappa_1 \int_0^t \| X(s) \|^2 \, ds). 
        \end{equation*}
        \item[{\crtcrossreflabel{\textbf{A2}}[ref:A2]}] There exists $\kappa_2> 0$, $\kappa_3\geq 0$ and a random variable $\Xi_\gamma$ such that 
        \begin{equation}
        \label{eq:A2}
            |X(t)|^2 + \kappa_2\int_0^t  \| X(s) \|^2 \, ds \leq  |x|^2 + \kappa_3 t + \Xi_\gamma \quad t\geq 0
        \end{equation}
        with $\kappa_0 > \kappa_1 \kappa_3/ \kappa_2 $ and 
        \begin{equation}
        \label{eq:A2 martingale estimate}
            \bP(\Xi_\gamma \geq R) \leq e^{-2\gamma R}, \quad R \geq 0. 
        \end{equation}
        \item[{\crtcrossreflabel{\textbf{A3}}[ref:A3]}] There exists a positive constant $c > 0$ such that for each $t\geq 0$ and $s\in [0,t]$
        \begin{equation*}
            |G(X(s), \tY(s))|^2 \leq c | X(s) - \tY(s) |^2.
        \end{equation*}
        \item[{\crtcrossreflabel{\textbf{A4}}[ref:A4]}] There exists a measurable function $V: \cH \to \R_+$ such that for some $\gamma_1>0$, $K>0$
        \begin{equation}
        \label{eq:ch4Lyapunovfnct2}
            \E V(X(t)) \leq \E V(X(s)) + \int_s^t \left( -\gamma_1 \E V(X(\tau)) + K \right)\; d\tau , \quad t\geq s\geq 0.
        \end{equation}
        Furthermore for any $M>0$ the function $x\mapsto |x|^2$ is bounded on the level sets $\lbrace V \leq M \rbrace$.
    \end{itemize}
%
Given $\kappa_1, \kappa_2$ from \ref{ref:A1} and \ref{ref:A2} respectively define the premetric $\theta_\alpha$ depending on a positive parameter $\alpha$ as
\begin{equation}
\label{eq:ch4 theta(u,v)}
    \theta_\alpha(x,y) := |x - y|^{2\alpha} e^{\alpha \upsilon |x|^2} \quad \text{with } \upsilon : = \frac{\kappa_1}{\kappa_2}
\end{equation}
    and, given $N\in \N$, define a semimetric $d_N$ as 
\begin{equation}
    \label{eq:ch4defdN}
    d_N (x,y) := N \theta_\alpha(x,y) \wedge N\theta_\alpha(y,x) \wedge 1 .
\end{equation}
 We will see in \autoref{thm:exponentialergodicity} that, thanks to~\nameref{asuA}, there exists $\alpha_0$ such that for all $\alpha\in (0, \alpha_0)$ the conditions of \autoref{thm:general_harris} are satisfied with respect to the associated semimetric $d_N$ for large enough $N$.  
First though we use \nameref{asuA} to demonstrate an important bound on the control term $G$ in Equation~\eqref{eq:ch4equationYtilde}. This will be exploited in~\autoref{lemma:dTVBM} which we take from~\cite{butkovsky2020}.
\begin{lemma}
    \label{prop:ch4generalgirsanov}
        Let the processes $X$ and $\tY$ be the solutions of Equation~\eqref{eq:ch4generalSDE} resp.\ of the controlled equation \eqref{eq:ch4equationYtilde}.
        If Assumptions \ref{ref:A1}-\ref{ref:A3} hold, then for all $t\geq 0$
        \begin{equation}\label{eq:ch4boundMdelta}
             \int_0^t \!  | Q_n^{-1/2} G(X(s),\tY(s))|^2 \, ds 
             \leq \tfrac{c\| Q_n^{-1/2} \|^2}{\chi} |x - y|^2 e^{\tfrac{\kappa_1}{\kappa_2}\left( |x|^2 + \Xi_\gamma \right)}\left( 1 - e^{-\chi t}\right)
         \end{equation}
        for some $\chi>0$.
    \end{lemma}
    \begin{proof}
        By \ref{ref:A3} and \ref{ref:A1} we have that for any $t\geq 0$ and $s\in [0,t]$
        \begin{equation*}
            |G(X(s),\tY(s))|^2 \leq c |x - y|^2 \exp(- \kappa_0 s + \kappa_1 \int_0^s \| X(\tau) \|^2 \, d\tau),
        \end{equation*}
        and by \ref{ref:A2} 
        \begin{equation*}
            |G(X(s),\tY(s))|^2 \leq c |x - y|^2 \exp(- \left(\kappa_0 - \tfrac{\kappa_1 \kappa_3}{\kappa_2}\right) s + \tfrac{\kappa_1}{\kappa_2}\left( |x|^2 + \Xi_\gamma \right)).
        \end{equation*}
        It follows that for any $ t\geq 0$
        \begin{equation*}
        \begin{split}
             \int_0^t  | Q_n^{-1/2} G(X(s),\tY(s))|^2 \, ds  \leq \| Q_n^{-1/2} \|^2 \int_0^t  | G(X(s),Y(s))|^2 \, ds \\
             \leq \tfrac{c\| Q_n^{-1/2} \|^2}{\chi} |x - y|^2 \exp(\tfrac{\kappa_1}{\kappa_2}\left( |x|^2 + \Xi_\gamma \right))\left( 1 - e^{-\chi t }\right)
            \end{split}
        \end{equation*}
        where $\chi = \kappa_0 - \tfrac{\kappa_1 \kappa_3}{\kappa_2}>0$. 
        \end{proof}
Given $n$ and $Q_n$ as in \nameref{asuA}, set $W_n:=\Pi_n W_X$, which is a Wiener process with covariance matrix $Q_n$ and independent of $(1 - \Pi_n)W_X$. A consequence of~\autoref{prop:ch4generalgirsanov} is that the process 
    \begin{equation*}
        \tW_n(t) := W_n(t) + \int_0^t G(X(s), Y(s)) \, ds
    \end{equation*}
    is absolutely continuous with respect to $W_n$, by Girsanov's theorem. Then $\tW_X$ in Equation~\eqref{eq:ch4tildeWn}, i.e.~$\tW_X=\tW_n + (1 - \Pi_n)W_X$ is a $Q$-Wiener process absolutely continuous with respect to $W_X$, hence also with respect to $W_Y$ as $\Law W_X = \Law W_Y$. 

We are ready to state the main theorem of this section.
    \begin{theorem}
        \label{thm:exponentialergodicity}
        Consider $X(t)$ solution of \eqref{eq:ch4generalSDE} with associated Markov semigroup $\cP_t$ which is Feller in $(\cH, |\cdot|)$ and \nameref{asuA} holds. Given $\kappa_1, \kappa_2$ and $\gamma$ as in \ref{ref:A1} and \ref{ref:A2}, set
        \begin{equation}
        \label{eq:QGupsilonalpha0}
         \upsilon = \frac{\kappa_1}{\kappa_2} \quand \alpha_0 =  \frac{1}{2}\wedge \frac{2\gamma}{\upsilon + 2\gamma}.
        \end{equation} 
        Then there exists a unique invariant measure $\mu_*$, and there exists $t>0$ and $\rho\in (0,1)$ such that 
        \begin{equation}
        \label{eq:SPDEsHarris}
            W_{\td}(P_t(x, \cdot), P_t(y,\cdot)) \leq \rho \,\td (x,y).
        \end{equation}
        where $\td(x,y)^2 = d_N(x,y)\left( 1 + V(x) + V(y) \right)$ and $d_N$ is defined as in \eqref{eq:ch4defdN} for $\upsilon$ and $\alpha_0$ as above.
%
        %
    \end{theorem}
    In order to prove this result we introduce first the following three lemmas.
    The first one is the so-called comparison theorem, a classic generalisation of the integral Gronwall lemma:
    \begin{lemma}[Comparison Theorem]
    \label{lemma:comparison_thm}
     Let $f$ be a continuous function, for which
        \begin{equation}
        \label{eq:A_proofeta}
            f(s) - f(r) \leq -\gamma \int_r^s f(\tau) \, d\tau + K(s -r) \fa 0< r< s
        \end{equation}
        with $\gamma, K >0$ then 
        \begin{equation}
            f(t) \leq f(0)e^{-\gamma t} + K/\gamma \fa t\geq 0.
        \end{equation}
    \end{lemma}
%
  The next lemma gives upper bounds for the total variation distance between a finite dimensional Wiener process and the same process with added drift. 
     \begin{lemma}[{\cite[Theorem~A.5]{butkovsky2020}}]
        \label{lemma:dTVBM}
            Let $B$ be a n-dimensional Wiener process with covariance operator $\sigma$, $(h(t))_{t\geq 0}$ a progressively measurable n-dimensional process and define
            \begin{equation*}
                \tilde{B}(t) = \int_0^t h(s) \, ds + B(t)\quad t\geq 0.
            \end{equation*}
            Fix $t>0$, then, if for some $\delta\in (0,1)$
            \begin{equation}
            \label{eq:lemmaMdelta}
                M_\delta : = \E \left( \int_0^t | \sigma^{-1/2} h(s) |^2\, ds \right)^\delta < \infty
            \end{equation}
             the following bounds hold:
            \begin{align}
                d_{TV}(\Law(B(s))_{s\leq t}, \Law(\tB(s))_{s\leq t}) &\leq 2^{(1- \delta) /(1+ \delta)}M_\delta^{1/(1+ \delta)}; \label{eq:dTVBM1}\\
                d_{TV}(\Law(B(s))_{s\leq t}, \Law(\tB(s))_{s\leq t}) &\leq 1 - \frac{1}{6}\min\left(\frac{1}{8}, \exp(-(2^{2- \delta}M_{\delta })^{1/\delta})\right).\label{eq:dTVBM2}
            \end{align}
        \end{lemma}
    The third lemma provides an estimate for the distance of two transition probabilities with different initial condition which will be the first block in building the proof of \autoref{thm:exponentialergodicity}. 
  To state the lemma, we set $n \in \N$ as in~\nameref{asuA} and recall the definition
    \begin{equation}\label{eq:defWn}
    W_n = \Pi_n W_X
    \qquad \text{and} \qquad
    \tW_n = \Pi_n \tW_X.
    \end{equation}
    \begin{lemma}
    \label{lemma:boundWdN}
        Let $X(t)$ and $Y(t)$ be the solutions of \eqref{eq:ch4generalSDE} and \eqref{eq:ch4generalSDEII} respectively and $P_t(x, \cdot), P_t(y, \cdot)$ their respective laws. If \nameref{asuA} holds, then there exists positive constants $C_{\Xi}$ and $\alpha_0$ such that for all $\alpha \in (0, \alpha_0)$ and $N\in \N$ we have
             \begin{multline}
             \label{eq:ch4lemma}
                W_{d_N}(P_t(x, \cdot), P_t(y, \cdot))
                \leq d_{TV}(\Law (W_n(t))_{s\leq t}, \Law (\tW_n(t))_{s\leq t})
                \\ + N C_\Xi \theta_\alpha(x,y) e^{- \chi\alpha t} ,
            \end{multline}
            for all $x,y\in \cH$.
    \end{lemma}
    
    \begin{proof}
        As shown in \cite[Theorem~2.4]{butkovsky2020}, by means of the classic coupling lemma and the gluing lemma (see e.g.~\cite{Villani}) it can be shown that 
        \begin{equation}
        \label{eq:jb2}
        W_{d_N}(P_t(x, .), P_t(y, .))
            \leq \underbrace{ d_{TV}(P_t(y, \cdot), \Law \tY(t))}_{(I)} 
            + \underbrace{\E[d_N(X(t), \tY(t))]}_{(II)}.
        \end{equation}
       We look at the two terms separately starting from (I). From the theory of stochastic differential equations the solutions of \eqref{eq:ch4generalSDEII} and \eqref{eq:ch4equationYtilde} can be seen as image via a measurable function $\Phi^y$ of their driving noise, $W_Y$ and $\tW_X$, respectively i.e.~$ Y(t) = \Phi^y((W_Y(s))_{s\leq t})$ and $\tY = \Phi^y((\tW_X(s))_{s\leq t})$.
      Then, using a classical property of the total variation norm, we get
        \begin{equation*}
        d_{TV}(\Law Y(t), \Law \tY(t))
            \leq d_{TV}(\Law (W_{Y}(s))_{s\leq t},\Law (\tW_{X}(s))_{s\leq t}).
        \end{equation*}
        By the coupling lemma there exists a coupling 
        %
        $(\xi, \tilde{\xi})$ of the laws of the finite dimensional Wiener processes $(W_n(s))_{s\leq t}$ and $(\tW_n(s))_{s\leq t}$ such that 
        \begin{equation}
        \label{eq:Wncoupling}
            d_{TV}(\Law (W_n(s))_{s\leq t},\Law (\tW_n(s))_{s\leq t}) = \bP ( \xi \neq \tilde{\xi} ). 
        \end{equation}
        %
        Therefore $B := \xi + (1 - \Pi_n)W_X$ and $\tilde{B}= \tilde{\xi} + (1 - \Pi_n)W_X$ provide a coupling for $(W_Y, \tW_X)$.
        Hence
        \begin{equation}
        \label{eq:Wcoupling}
            d_{TV}(\Law (W_{Y}(s))_{s\leq t},\Law (\tW_{X}(s))_{s\leq t})
            \leq \bP( B \neq \tilde{B} ) .
        \end{equation}
        But since $B = \tilde{B} \Leftrightarrow \xi = \tilde{\xi}$ we can use Equation~\eqref{eq:Wncoupling} in  Equation~\eqref{eq:Wcoupling} and obtain
        %
        %
        \begin{equation*}
             d_{TV}(\Law (W_{Y}(s))_{s\leq t},\Law (\tW_{X}(s))_{s\leq t})
             \leq  d_{TV}(\Law(W_n(s))_{s\leq t},\Law(\tW_n(s))_{s\leq t})
        \end{equation*}
        giving the first part of \eqref{eq:ch4lemma}.
    
        Next we look at~(II) in \eqref{eq:jb2}. By definition of $d_N$ \eqref{eq:ch4defdN} we have 
        \begin{equation}
        \label{eq:g2}
            \E\,  d_N(X(t), \tY(t)) \leq N \E \, \theta_\alpha(X(t),\tY(y)) .
        \end{equation}
        Combining Assumption~\ref{ref:A1} and \ref{ref:A2} we have that 
        \begin{equation}
            \label{eq:ch4proofbutkA1+A2}
            |X(t)- \tY(t)|^2 \leq |x-y|^2 \exp(- \chi t + \upsilon (|x|^2 - |X(t)|^2 + \Xi_\gamma))
        \end{equation}
        namely, for all $\alpha>0$
        \begin{equation*}
            |X(t)- \tY(t)|^{2\alpha} e^{\alpha \upsilon |X(t)|^2}\leq |x-y|^{2\alpha} e^{\alpha \upsilon|x|^2} \exp(- \alpha\chi t + \alpha \upsilon \Xi_\gamma).
        \end{equation*}
        Therefore by the definition of the premetric $\theta_\alpha$ \eqref{eq:ch4 theta(u,v)}
        \begin{equation}
        \label{eq:ch4proofbutkcontractII}
             \E\, \theta_\alpha(X(t) , \tY(t)) \leq C_\Xi \theta_\alpha(x,y) \exp(- \chi\alpha t) 
        \end{equation}
        where by \eqref{eq:A2 martingale estimate}
        \begin{equation}
         \label{eq:ch4CXi}
          C_\Xi := \E \,  \exp(\upsilon \alpha \Xi_\gamma ) < \infty \quad \text{if } \upsilon \alpha < 2\gamma.
        \end{equation}
        
        Putting together these results in \eqref{eq:jb2} we have that \eqref{eq:ch4lemma} holds for all $\alpha \in (0, \alpha_0)$ setting $\alpha_0 := 2\gamma /\upsilon$.
        
    \end{proof}

    We are now ready to show \autoref{thm:exponentialergodicity}:
    
    \begin{proof}[Proof of \autoref{thm:exponentialergodicity}]
        
        %
%
First we show that $V$ is a Lyapunov function as in \eqref{def:LyapunovfctHMS}. By \ref{ref:A4} we know that there exists $\gamma_1, K$ strictly positive such that 
         \begin{equation}
        \label{eq:lyapunov P_t proof}
             \cP_t V(x) -\cP_s V(x) \leq  \int_s^t \left( -\gamma_1 \cP_\tau V(x) + K \right)\, d\tau , \quad t\geq s\geq 0.
         \end{equation}
        Therefore the hypothesis of \autoref{lemma:comparison_thm} are satisfied by the function $f(t) :=  \cP_t V(x)$, which has non-negative values and is continuous in time, thus we have the desired result
        \begin{equation}
        \label{eq:ch4lyapunovSDE}
              \cP_t V(x) \leq  e^{-\gamma_1 t}V(x) + K_V
        \end{equation}
        with $K_V = K/\gamma_1$.
        Our result will follow from the Generalised Harris' Theorem~\ref{thm:general_harris} if we are able to show that for all $\alpha \in (0, \alpha_0)$, there exists $N_*\in \N$ and $t_* \in \R_+$ such that for all $t>t_*$ and $N\in \N$ such that $N > N_*$: 
        \begin{enumerate}[label=(\roman*)]
        \item\label{item:contracting} the semimetric $d_N$ \eqref{eq:ch4defdN} is contracting for $\cP_t$; 
        \item\label{item:small} the level set $\lbrace x \in \cH\, : \, V(x) \leq 4K_V \rbrace$ is $d_N$-small. 
        \item \label{item:metric}There exists a complete metric $d_0$ such that $d_0\leq \sqrt{d_N}$, and further $\cP_t$ is Feller on $(\cH, d_0)$.
        \end{enumerate}
        %
        %
        We start by showing item~\ref{item:metric}.
        The Hilbert space $\cH$ is endowed with the norm $|\cdot|$ associated to its scalar product. Define $d_0(x,y) = |x - y|\wedge 1$. This is a complete metric that defines the same topology as the norm $|\cdot|$, and as $\cP_t$ is Feller in $(\cH, |\cdot|)$, it is in particular Feller in $(\cH, d_0)$. Furthermore, we have 
        \begin{equation*}
           \sqrt{ d_N(x,y)} \geq N^{1/2} |x- y|^\alpha \wedge 1 \geq |x- y|\wedge 1 =d_0(x,y)
        \end{equation*}
        as $\alpha\in (0, \alpha_0)$ with $\alpha_0\leq 1/2$.
        The remainder of the proof is dedicated to show the properties \ref{item:contracting} and \ref{item:small}.
    \paragraph{\ref{item:contracting} The semigroup $\cP_t$ is $d_N$--contracting.} 
        We have to show that there exists $\rho<1$ such that 
        \begin{equation*}
            W_{d_N}(P_t(x, \cdot), P_t(y, \cdot)) \leq \rho \, d_N(x, y)
        \end{equation*}
        for all $x, y$ such that $d_N(x, y) < 1$.
        Thanks to \autoref{lemma:boundWdN} we have that for all $\alpha< 2\gamma/\upsilon$
        \begin{multline}
        \label{eq:boundWdN}
             W_{d_N}(P_t(x, .), P_t(y, .)) \leq d_{TV}(\Law (W_n(s))_{s\leq t},\Law (\tW_n(s))_{s\leq t})\\
                + N C_\Xi \theta_\alpha(x,y) e^{- \chi\alpha t} 
        \end{multline}
        We want to apply \autoref{lemma:dTVBM} to bound the first term on the right hand side so we have to ensure that \eqref{eq:lemmaMdelta} holds. 
       In \eqref{eq:ch4tildeWn} we defined
        \begin{equation*}
            \tilde{W}_X(t) = W_X(t) + \int_0^t G(X, \tY)\, ds
        \end{equation*}
        and in particular, since $G$ is $n$-dimensional, we have 
          \begin{equation*}
            \Pi_n\tilde{W}_X(t) = \Pi_n W_X(t) + \int_0^t G(X, \tY)\, ds.
        \end{equation*}
        Then, using $Q_n$ covariance operator of the Wiener process $W_n = \Pi_n W_X$ as in \eqref{eq:covarianceQn}, in order to apply \autoref{lemma:dTVBM} we have to show that 
        \begin{equation}
        \label{eq:ch4ourMdelta}
            M_\delta  = \E \left(\int_0^t |Q_n^{-1/2}G(X, \tY)|^2 \, ds \right)^\delta 
        \end{equation}
      is finite. By \autoref{prop:ch4generalgirsanov} one has
        \begin{equation*}
             \int_0^t |Q_n^{-1/2}G(X, \tY)|^2 \, ds \leq \tfrac{c\| Q_n^{-1/2} \|^2}{\chi} |x - y|^2 \exp(\upsilon\left( |x|^2 + \Xi_\gamma \right))\left( 1 - e^{-\chi t }\right)
        \end{equation*}
        where 
        \begin{equation*}
            \upsilon := \frac{\kappa_1}{\kappa_2} \quand \chi = \kappa_0 - \upsilon \kappa_3>0.
        \end{equation*}
       Since Equation \eqref{eq:A2 martingale estimate} holds for $\Xi_\gamma$ we have that 
        \begin{equation*}
          \E \,  \exp(\upsilon \delta \Xi_\gamma ) < \infty \quad \text{if } \upsilon \delta < 2\gamma
        \end{equation*}
         and so for all $0<\delta < \left(2\gamma / \upsilon\right) \wedge 1$
         \begin{equation}
         \label{eq:boundourMdelta}
             M_\delta \leq \tilde{C}_\delta |x - y|^{2\delta} \exp(\upsilon \delta |x|^2) \quad \text{with } \tilde{C}_\delta =  \left(\tfrac{c\| Q_n^{-1/2} \|^2}{\chi}\right)^\delta \E \, e^{\upsilon \delta \Xi_\gamma }.
         \end{equation}
         Therefore, since condition \eqref{eq:lemmaMdelta} holds, the bound \eqref{eq:dTVBM1} in \autoref{lemma:dTVBM} and \eqref{eq:boundourMdelta} give
        \begin{align*}
            d_{TV}(\Law (W_n(s))_{s\leq t},\Law (\tW_n(s))_{s\leq t})\leq 2^{(1- \delta) /(1+ \delta)} M_\delta^{\frac{1}{1+ \delta}}\\
            \leq C_{\frac{\delta}{1 + \delta}} \left(|x - y|^{2} \exp(\upsilon |x|^2)\right)^{\frac{\delta}{1 + \delta}}
        \end{align*}
        where $C_{\frac{\delta}{1 + \delta}} = 2^{(1- \delta) /(1+ \delta)}\tilde{C}_\delta$. Since $\delta \in (0, \left(2\gamma / \upsilon\right) \wedge 1)$, the exponent $ \delta/(1+ \delta) = :\alpha $ is in the interval
        \begin{equation*}
          0< \alpha < \alpha_0 = \frac{1}{2}\wedge \frac{2\gamma}{\upsilon + 2\gamma} < \frac{2\gamma}{\upsilon}.
        \end{equation*}
        %
        Given the definition of the premetric $\theta_\alpha$ \eqref{eq:ch4 theta(u,v)} we have then shown that for all $\alpha \in (0, \alpha_0)$ there exists $C_\alpha>0$ such that
         \begin{equation*}
             d_{TV}(\Law (W_n(s))_{s\leq t},\Law (\tW_n(s))_{s\leq t}) \leq C_\alpha \theta_{\alpha} (x,y) \fa x,y\in \cH, \; x \neq y.
        \end{equation*}
        Going back to \eqref{eq:boundWdN}, we have proved that
        \begin{align*}
            W_{d_N}(P_t(x, \cdot\,), P_t(y, \cdot\,)) \leq  C_\alpha \theta_\alpha(x,y) + N C_{\Xi} e^{- \chi\alpha t }\theta_\alpha(x,y) \\
            =  N\theta_\alpha(x,y)\left( C_\alpha N^{-1} + C_\Xi e^{- \chi\alpha t } \right)
        \end{align*}
        and, inverting the roles of $x$ and $y$, we get in the same way
        \begin{equation*}
             W_{d_N}(P_t(x, \cdot\,), P_t(y, \cdot\,)) \leq N\theta_\alpha(y,x)\left( C_\alpha N^{-1} + C_\Xi e^{- \chi\alpha t } \right).
        \end{equation*}
        Therefore
        \begin{equation}
        \label{eq:boundWdN2}
             W_{d_N}(P_t(x, \cdot\,), P_t(y, \cdot\,)) \leq \left(N\theta_\alpha(x,y)\wedge N \theta_\alpha(y,x)\right) \left( C_\alpha N^{-1} + C_\Xi e^{- \chi\alpha t } \right)
        \end{equation}
        and if $d_N(x,y)<1$ we have that 
         \begin{equation*}
             W_{d_N}(P_t(x, \cdot\,), P_t(y, \cdot\,)) \leq d_N(x,y) \left( C_\alpha N^{-1} + C_\Xi e^{- \chi\alpha t } \right).
        \end{equation*}
        Then for all $N\in \N$ and $t>0$ such that 
        \begin{equation}
        \label{eq:ch4proofbutkconditionNt}
             \rho:=  C_\alpha N^{-1} + C_\Xi e^{- \chi\alpha t }  < 1
        \end{equation} 
        we showed that $\cP_t$ is $d_N$-contracting. 
  
        \paragraph{\ref{item:small} The sublevel set of $V$ is $d_N$-small.} We have now to prove that there exists $\varepsilon>0$ such that 
        \begin{equation}
        \label{eq:small}
            W_{d_N}(P_{t}(x, \cdot),P_{t}(y, \cdot)) \leq 1 - \varepsilon
        \end{equation}
        for all $x, y \in \lbrace z\in \cH \, : \, V(z) \leq 4K_V \rbrace$.
      Although a closer examination of the previous step will show that \eqref{eq:boundWdN2} holds in fact for all $x,y$ (not only for $x, y$ such that $d_N(x, y) < 1$), there is no way to obtain~\eqref{eq:small} directly from~\eqref{eq:boundWdN2}. 
      Despite $\theta_\alpha(x,y)$ being bounded over the sublevel set $\lbrace V \leq 4K_V \rbrace$ (see Eq.~\ref{eq:ch4proofbutkboundthetaLyap} below), as soon as $\theta_\alpha(x,y) \wedge \theta_\alpha(y,x) >0$, we would have to make $C_\alpha + N C_\Xi e^{- \chi\alpha t }$ arbitrarily small to have \eqref{eq:small}, which is impossible.
        However we can use Equation~\eqref{eq:dTVBM2} in \autoref{lemma:dTVBM} to arrive at a different estimate for $ d_{TV}(\Law (W_n(s))_{s\leq t},\Law (\tW_n(s))_{s\leq t})$ in \eqref{eq:boundWdN}.
%
        Using that result,
        we have that 
        \begin{equation*}
          d_{TV}(\Law (W_n(s))_{s\leq t},\Law (\tW_n(s))_{s\leq t}) \leq 1 - \frac{1}{6}\min\left(\frac{1}{8}, \exp(-(2^{2- \delta}M_{\delta })^{1/\delta})\right),
        \end{equation*}
        where $M_\delta$ is as in \eqref{eq:ch4ourMdelta}. 
        Since, by Assumption~\ref{ref:A4}, $z\mapsto |z|^2$ is bounded over $\lbrace V \leq 4K_V \rbrace$, so is $\theta_\alpha(x,y)$ i.e.~there exists $C_K>0$ such that
        \begin{equation}
            \label{eq:ch4proofbutkboundthetaLyap}
            \theta_\alpha(x, y) = |x - y|^{2\alpha}e^{\alpha \upsilon|x|^2} < C_K 
        \end{equation} 
        for all $x, y\in \lbrace V\leq 4K_V \rbrace$.  
         Thanks to \eqref{eq:boundourMdelta} and \eqref{eq:ch4proofbutkboundthetaLyap} it follows that
        \begin{equation*}
             M_\delta \leq \tilde{C}_\delta |x - y|^{2\delta} \exp(\upsilon \delta |x|^2) \leq \tilde{C}_\delta C_K
        \end{equation*}
        for all $0< \delta< (2\gamma/\upsilon)\wedge 1$. 
        Setting $\varepsilon_1(\delta)$ to be 
        \begin{equation*}
        \label{eq:ch4epsilon1}
            \varepsilon_1 = \frac{1}{6}\min\left(\frac{1}{8}, \exp(-(2^{2- \delta} \tilde{C}_\delta C_K)^{1/\delta})\right),
        \end{equation*}
        we have that 
        \begin{equation}
        \label{eq:ch4proofbutksmallpt2}
          d_{TV}(\Law (W_n(s))_{s\leq t},\Law (\tW_n(s))_{s\leq t}) \leq 1 - \varepsilon_1.
        \end{equation}
        Then from \eqref{eq:boundWdN} i.e.~
        \begin{equation*}
             \Wtd{P_{t}(x, \cdot)}{ P_{t}(y, \cdot)} \leq d_{TV}(\Law (W_n(s))_{s\leq t},\Law (\tW_n(s))_{s\leq t}) + N  C_\Xi e^{- \alpha \chi t} \theta_\alpha(x,y)
        \end{equation*}
        by \eqref{eq:ch4proofbutkboundthetaLyap} and \eqref{eq:ch4proofbutksmallpt2} we have
        \begin{equation*}
            \Wtd{P_{t}(x, \cdot)}{ P_{t}(y, \cdot)} \leq  1 - \varepsilon_1+ N C_K C_\Xi e^{- \chi\alpha t} .
        \end{equation*}
        %
        Then for all $N\in \N$ and $t>0$ such that %
        \begin{equation}\label{eq:ch4proofbutkconditionNt2}
            \varepsilon_1 - N C_K C_\Xi e^{- \chi\alpha t } = : \varepsilon > 0
        \end{equation} 
        we showed that the level set $\lbrace x\in \cH \, : \, V(x) \leq 4K_V \rbrace$ is $d_N$-small. 
%
The proof is completed by noting that there are $N_*, t_*$ so that conditions~\eqref{eq:ch4proofbutkconditionNt} and~\eqref{eq:ch4proofbutkconditionNt2} are satisfied simultaneously.
%
%
    \end{proof}
%
    Thanks to the result just shown one has a precise formulation of the spectral gap property \autoref{cor:spectralgap}, crucial ingredient to develop response theory. In fact by \autoref{thm:exponentialergodicity} and \autoref{thm:general_harris} we know that such a distance--like function is $\td(x,y)^2 = d_N(x,y)(1 + V(x)+ V(y)) $, with $d_N$ as in \eqref{eq:ch4defdN} and $V$ a Lyapunov function of the system. Note that, by definition, the distance-like function $d_N$ is comparable to the $\alpha_0$--power of the original metric on the space, with $\alpha_0$ as in \autoref{thm:exponentialergodicity}. Therefore $\cP_t$, the semigroup associated to the solution of \eqref{eq:ch4generalSDE}, exhibits a spectral gap on the set of observables which are $\alpha_0$--H\"{o}lder continuous over the level sets of the Lyapunov function $V$.
%
\section{Spectral Gap for the stochastic 2LQG model}
\label{sec:exp_stab_QG}
Let $\q(t, \q_0)$ be the solution of the stochastic two--layer quasi--geostrophic model \eqref{eq:QG_stoc_vec} on the Hilbert space $\cH =(\Hmo, \vertiii{\cdot}_{-1})$ and $\cP_t$ the associated semigroup i.e.~
\begin{equation*}
     \cP_t \varphi (\q_0) = \E \, \varphi(\q(t, \q_0)) \fa \varphi \in B_b(\cH).
\end{equation*}
The semigroup is Feller in the metric introduced by $\vertiii{\cdot}_{-1}$ since the solution can be shown to be continuous in the initial conditions with respect to the metric induced by $\vertiii{\cdot}_{-1}$ (see \cite{thesis}).
Then we want to ensure that this system has exponential ergodicity as described in \autoref{sec:method}, namely there exists a unique invariant measure $\mu_*$, there exists $C>0$, $\gamma>0$ such that 
 \begin{equation*}
     W_{\td}\left( P_t(x, \cdot), \mu_*\right) \leq C\left(1 + V(x)\right) e^{-\gamma t} \fa x\in \cH, \, t\geq 0
 \end{equation*}
and $\cP_t$ has a spectral gap i.e.~there exists $\rho<1$ such that 
\begin{equation*}
    \| \cP_t \varphi - \langle \varphi, \mu_*\rangle \|_{\td} \leq \rho \| \varphi - \langle \varphi, \mu_*\rangle \|_{\td} 
\end{equation*}
for all $\varphi: \cH \to \R$ with $\| \varphi \|_{\td}<\infty$. We consider $\td$ as in \autoref{thm:general_harris} and $d_N$ as in \eqref{eq:ch4defdN}, now in $\cH = \left( \Hmo, \vertiii{\cdot}_{-1}\right)$, namely
    \begin{align}
    \begin{split}
          \td(\x, \y)^2 &= d_N(\x, \y) (1 + V(\x) + V(\y))\quad \text{with}\\
            d_N(\x, \y) &= N \theta_\alpha (\x, \y) \wedge  N \theta_\alpha (\x, \y) \wedge 1 \quad \text{and}\\
            \theta_\alpha (\x, \y) &= e^{\alpha \upsilon \vertiii{\x}_{-1}^2} \vertiii{\x - \y}_{-1}^{2\alpha}, \quad \alpha \in (0, \alpha_0)
    \end{split}\label{eq:td_dN_theta_QG}
    \end{align}
 for $\x, \y\in \Hmo$ and an appropriate choice of the parameters $N$, $\upsilon$ and $\alpha_0$ given in the following main theorem of this section. 
 Further, $\{\lambda_n\}_{n \in \N}$ is an increasing sequence of eigenvalues of $-\Delta$.
   \begin{theorem}
    \label{thm:exp_stab_QG}
         Given the stochastic 2LQG model \eqref{eq:QG_stoc_vec}, there exists $r_0$ (depending on $\nu, Q$ and the forcing $f$, see Eq.~\eqref{eq:condition_r}) with the following properties:
         Suppose that 
         \begin{enumerate}[label=(\roman*)]
             \item $r > r_0$, and
             \item there exists $n\in \N$ such that $\Pi_n \cH  \subset \range Q$, and 
             \item $\nu - 2r \lambda_n^{-1}>0$. 
         \end{enumerate}
        Then there exists a unique invariant measure $\mu_*$ as well as $t>0$ and $\rho <1$ such that 
        \begin{equation}
            \Wtd{P_t(\q_0, \cdot)}{P_t(\tq_0, \cdot)} \leq \rho\, \td(\q_0, \tq_0) 
        \end{equation}
        for all $\q_0, \tq_0\in \cH$. Here $\td$ is as in \eqref{eq:td_dN_theta_QG} with Lyapunov function $V(\x) = \vertiii{\x}_{-1}^2$ and parameters
        \begin{equation*}
           \upsilon = \frac{k_B}{\nu - \frac{2 \gamma \Tr Q}{\lambda_1^2}} \quand \alpha_0 = \frac{1}{2} \wedge \frac{2 \gamma }{\upsilon + 2\gamma},
        \end{equation*}
        where $0<\gamma< \lambda_1^2 \nu /2 \Tr Q$, $k_B := k_0^2/2\nu$ and $k_0$ as in \autoref{lemma:ch1propertiesB}.
    \end{theorem}
    
    \begin{proof}
        As the semigroup $\cP_t$ is Feller in $\cH$ we only have to ensure that \nameref{asuA} is satisfied, as then \autoref{thm:exponentialergodicity} gives the desired result.
        Consider, beside the original equation \eqref{eq:QG_stoc_vec}, the controlled system 
    \begin{align}
    \label{eq:QG_controlled}
     \begin{split}
         &d\tq + \left(B(\tpsi, \tpsi) + \beta \partial_x\tpsi\right)\, dt = \nu \Delta^2\tpsi \;dt+  \binom{f + G(\q, \tq)}{- r \Delta \tilde{\psi}_1}\, dt + d\bm{W}\\
         &\tq = (\Delta + M) \tpsi
     \end{split}
    \end{align}
    with initial condition $\tq (0) = \tq_0 \neq \q_0$, $\bm{W} = (W, \, 0)^t$, and control $G$ to be determined later. 
    Define the difference variables
    \begin{equation}
        \bfxi = \q - \tq \quand\bfphi = \bfpsi - \tilde{\bfpsi}.
    \end{equation}
    Then $\bfxi = (\Delta + M)\bfphi $ satisfies the equation
    \begin{equation}
        \label{eq: error dynamic}
        \dv{\bfxi}{t} + B(\bfphi, \tilde{\bfpsi}) + B(\bfpsi, \bfphi) = \nu \Delta^2 \bfphi - \binom{G(\q, \tq)}{r\Delta \phi_2} 
    \end{equation}
    with initial condition $\bfxi_0 = \q_0 - \tq_0\neq 0$.
    
     Let $\{\lambda_k\}_{k \in \N}$ be an increasing sequence of eigenvalues of $-\Delta$ with corresponding eigenvectors $e_k$ forming an orthonormal basis for $\cH$. Consider the following finite dimensional control 
    \begin{equation}
        \label{eq:ourcontrol}
        G(\q, \tq) =  a \Pi_n (\Delta \psi_1 - \Delta \tilde{\psi}_1)=  a  \Pi_n \Delta \phi_1 \\
    \end{equation}
    where $\Pi_n$ is the projection onto $\cH_n= \text{span}\lbrace e_k, \; k = 1, \ldots n \rbrace$ and $a >0 $ is a parameter to be found below. 
    %
    %
   The controlled system is well posed in the sense of \autoref{thm:solutions} as it can be treated effectively as the uncontrolled system. In fact the control can be split into two parts with finite rank, a lower level perturbation of the viscosity term and an additional time dependent forcing which is continuous in time.

    \paragraph{Proof of A1.} 
    Taking the $\bLtwo$ scalar product of \eqref{eq: error dynamic} with $\bfphi$, we obtain 
    \begin{equation*}
        \left(\dv{\bfxi}{t}, \bfphi\right) + (B(\bfpsi, \bfphi), \bfphi) = \nu |\Delta \bfphi|^2 - h_1(a\Pi_n \Delta \phi_1, \phi_1) + r h_2\|\phi_2\|^2,
    \end{equation*}
    where we have used the fact that $(B(\bfxi, \bfpsi),\bfxi) = 0$ of \eqref{eq: (B(U,V), U) = 0}. 
    Recall that by \eqref{eq:ch1(u,v)=|u|*} we have
    \begin{align*}
         \left(\dv{\bfxi}{t}, \bfphi\right) =  \dfrac{\mathrm{d}}{\mathrm{d}t}(\bfxi, \bfphi) - \left(\bfxi, \dv{\bfphi}{t}\right) = - \dv{\vertiii{\bfxi}_{-1}^2}{t} - \left( -\tA\bfphi, \dv{\bfphi}{t}\right)
    \end{align*}
    and since the operator $\tA$ is self-adjoint
    \begin{align*}
        \left(\dv{\bfxi}{t}, \bfphi\right) &=  - \dv{\vertiii{\bfxi}_{-1}^2}{t} - \left( \bfphi, \dv{(-\tA\bfphi)}{t}\right)\\
        2\left(\dv{\bfxi}{t}, \bfphi\right) &=   - \dv{\vertiii{\bfxi}_{-1}^2}{t}.
    \end{align*}
    Then we obtain the following equation
    \begin{equation*}
        \frac{1}{2}\dv{}{t}\vertiii{\bfxi}_{-1}^2 +\nu |\Delta \bfphi|^2 + rh_2\|\phi_2\|^2  =  (B(\bfpsi, \bfphi), \bfphi) + h_1(a \Pi_n \Delta \phi_1 , \phi_1) .
    \end{equation*}
   Denoting the orthogonal complement of $\Pi_n$ as $\oPi_n$, we can write the control term as \( G(\q, \tq) = a\Delta \phi_1 - a \oPi_n \Delta \phi_1\). One can consider the term $a \oPi_n \Delta \phi_1$ as an error term because we have to work with finite dimensional controls.
    Then, given 
    the generalised Poincar\'{e} inequalities for any $n\geq 1$ i.e.~
    \begin{equation*}
        \| \Pi_n \phi_1 \|_{k+1}^2 \leq \lambda_n \|\Pi_n \phi_1 \|_k^2 \quand \|\oPi_n \phi_1 \|_k^2 \leq \lambda_n^{-1} \|\oPi_n \phi_1 \|_{k+1} ^2,
    \end{equation*}
   we have an appropriate bound for the control term 
    \begin{align}
        (G(\q, \tq)), \phi_1) &= a (\Delta \phi_1 - \oPi_n\Delta \phi_1, \phi_1) \nonumber \\&= - a h_1 \| \phi_1\|^2 + a h_1\| \oPi_n \phi_1\|^2 \nonumber \\
        &\leq - a h_1 \| \phi_1\|^2 + a h_1\lambda_n^{-1} |\Delta \phi_1|^2 .\label{eq:bound control}
    \end{align}
    Given the estimate \eqref{eq:bound control} for the control, we get 
    \begin{equation*}
         \frac{1}{2}\dv{}{t}\vertiii{\bfxi}_{-1}^2 +\nu |\Delta \bfphi|^2 + rh_2\|\phi_2\|^2  \leq  (B(\bfpsi, \bfphi), \bfphi) - a h_1 \| \phi_1\|^2  + a h_1\lambda_n^{-1} |\Delta \phi_1|^2 .
    \end{equation*}
    By \eqref{eq:bound(B(u,u),v)} and \eqref{eq: (B(U,V), W)= - (B(W,V), U) } we know that 
    \begin{equation*}
        |(B(\bfpsi, \bfphi), \bfphi)| = |(B(\bfphi, \bfphi), \bfpsi)| \leq k_0 \| \bfphi \| |\Delta \bfphi| |\Delta \bfpsi|,
    \end{equation*}
    and by Young inequality, given $k_B = k_0^2/2\nu$ we have 
    \begin{equation*}
         |(B(\bfpsi, \bfphi), \bfphi)| \leq \tfrac{\nu}{2}|\Delta \bfphi|^2 +  k_B |\Delta \bfpsi|^2\| \bfphi \|^2 .
    \end{equation*}
    It follows that
    \begin{equation*}
        \dv{}{t}\vertiii{\bfxi}_{-1}^2 +\nu|\Delta \bfphi|^2 + 2a h_1 \| \phi_1\|^2+  2rh_2\|\phi_2\|^2  \leq 2k_B |\Delta \bfpsi|^2 \| \bfphi \|^2 + 2a h_1\lambda_n^{-1} |\Delta \phi_1|^2,
    \end{equation*}
    and in particular setting $a = r$, 
    \begin{equation}
    \label{eq:ch3QGenergyestimate}
        \dv{}{t}\vertiii{\bfxi}_{-1}^2 +\left(\nu - 2r \lambda_n^{-1} \right)|\Delta \bfphi|^2 \leq  \| \bfphi \|^2 \left( 2k_B |\Delta \bfpsi|^2 - 2r \right).
    \end{equation}
    Choosing $n$ so that 
    \begin{equation}
    \label{eq:condition_n}
        \nu - 2r \lambda_n^{-1} > 0,
    \end{equation}
    Gronwall's lemma gives that
    \begin{equation*}
        \vertiii{\bfxi(t)}_{-1}^2 \leq \vertiii{\bfxi(0)}_{-1}^2 \exp \left( - 2 r t + 2k_B \int_0^t|\Delta \bfpsi|^2  \, ds \right).
    \end{equation*}
    Assumption~\ref{ref:A1} follows immediately with constants \( \kappa_0 = 2r\) and \( \kappa_1 = 2k_B\).

\paragraph{Proof of A2.} 
     Consider the original model \eqref{eq:QG_stoc_vec} and let us apply It\^{o} formula to compute $ d\vertiii{\q}_{-1}^2$. Since $\q = - \tA\bfpsi$ we have 
    \begin{equation*}
     - d\vertiii{\q}_{-1}^2  = d(\q, \bfpsi) = - d\left( \q, \tA^{-1}\q \right) ,
    \end{equation*}
     and since $\tA$ is self-adjoint we have 
    \begin{equation*}
        d\left( \q, \tA^{-1}\q \right) = 2 (d \q, \tA^{-1}\q) + \Tr \left[(Q^{1/2})^* \tA^{-1} Q^{1/2}\right]\, dt.
    \end{equation*}
    %
    Therefore, setting 
     \begin{equation*}
        T_Q :=\Tr \left[(Q^{1/2})^* \tA^{-1} Q^{1/2}\right],
    \end{equation*}
    we have that $ - d\vertiii{\q}_{-1}^2 =  2(d \q, \bfpsi) -   T_Q\, dt$, which gives
    \begin{equation*}
         - d\vertiii{\q}_{-1}^2 = 2\left(\nu (\Delta^2\bfpsi, \bfpsi) + h_1(f, \psi_1) - r(\Delta \psi_2, \psi_2) - \tfrac{1}{2}T_Q \right)\, dt + 2 (\bfpsi, d\mathbf{W}) 
    \end{equation*}
    where we have used that $(B(\bfpsi, \bfpsi), \bfpsi) =0$ and $(\partial_x  \bfpsi, \bfpsi) = 0$.
    By Green's theorem and the definition of $\mathbf{W}$
    \begin{equation}
    \label{eq:ch2.3.1proof}
        d\vertiii{\q}_{-1}^2 = - 2\left(\nu |\Delta \bfpsi|^2 + h_1(f, \psi_1) + rh_2\|\psi_2\|^2 - \tfrac{1}{2}T_Q \right) \, dt - 2h_1\left( \psi_1, dW\right) .
    \end{equation}
    Next, using Cauchy-Schwartz, Young and Poincar\'{e} inequalities we can bound the deterministic forcing term as follows
    \begin{equation}
    \label{eq:forcing}
        -  2(f, h_1\psi_1) \leq  2| (f, h_1\psi_1) |  \leq  \tfrac{h_1}{\nu}\|f\|_{-2}^2 + \nu h_1|\Delta \psi_1|^2
    \end{equation} 
    and, using this estimate in \eqref{eq:ch2.3.1proof}, we have
    \begin{equation*}
       \vertiii{\q(t)}_{-1}^2 - \vertiii{\q_0}_{-1}^2 + \nu \int_0^t |\Delta \bfpsi|^2 \, ds + 2rh_2 \int_0^t\|\psi_2\|^2\, ds \leq \kappa_3 t  + 2h_1X_t
    \end{equation*}
     where $\kappa_3 = \tfrac{h_1}{\nu}\|f\|_{-2}^2 + T_Q $ and $X_t$ is defined by 
    \begin{equation*}
        X_t := \int_0^t (\psi_1(s), dW(s)).
    \end{equation*}
    The quadratic variation of this process is 
    \begin{equation*}
         \langle X\rangle_t := \int_0^t \|(\psi_1(s), \cdot) \|_{L_2^0}^2 \, ds = \int_0^t \sum_{k\in \N} |(\psi_1, Q^{1/2}e_k)|^2 \, ds .
    \end{equation*}
    By Cauchy-Schwartz inequality, this can be bounded by
    \begin{equation}
        \label{eq:boundquadvar_bfPoincare}
        \langle X\rangle_t \leq \int_0^t|\psi_1|^2 \sum_{k\in \N}  |Q^{1/2}e_k|^2 \, ds = \Tr Q \int_0^t  |\psi_1(s)|^2 \, ds .
    \end{equation}
   Then, since $X_t$ is a continuous martingale, it can be shown that for all $\gamma>0$ and $R>0$
    \begin{equation*}
        \bP \left( \sup_{t\geq 0} \left(X_t - \gamma \langle X \rangle_t\right) > R\right)\leq e^{- 2\gamma R}
    \end{equation*}
    hence, setting $\Xi_\gamma :=\sup_{t\geq 0} \left(X_t - \gamma \langle X \rangle_t\right)$,
    \begin{equation*}
        \E \exp(K \Xi_\gamma) < \infty \fa K < 2\gamma.
    \end{equation*}
    Therefore with a simple manipulation we get 
        \begin{multline}
        \label{eq:ch2energyestimate|q|*}
           \vertiii{\q(t)}_{-1}^2 - \vertiii{\q_0}_{-1}^2 + \nu \int_0^t  |\Delta \bfpsi|^2 \, ds + 2 r h_2\int_0^t \| \psi_2\|^2 \, ds  \leq \\ \kappa_3 t + 2h_1 \Xi_\gamma + 2h_1\gamma \langle X \rangle_t .
        \end{multline}
        Using Poincar\'{e} inequality twice in \eqref{eq:boundquadvar_bfPoincare} we obtain 
        \begin{equation*}
            \langle X\rangle_t \leq \frac{\Tr Q}{\lambda_1^2}\int_0^t |\Delta \psi_1|^2 \; ds 
        \end{equation*}
        and using this in \eqref{eq:ch2energyestimate|q|*} gives
        %
        \begin{equation*}
             \vertiii{\q(t)}_{-1}^2 - \vertiii{\q_0}_{-1}^2 + \left( \nu - \tfrac{2\gamma \Tr Q}{\lambda_1^2}\right)\int_0^t  |\Delta \bfpsi |^2 \, ds   - \kappa_3 t  \leq 2h_1 \Xi_\gamma .
        \end{equation*}
        Finally, Assumption~\ref{ref:A2} is satisfied with
        \begin{equation*}
            \kappa_2 =  \nu - \tfrac{2\gamma \Tr Q}{\lambda_1^2} \quand  \kappa_3 =  \tfrac{h_1}{\nu}\|f\|_{-2}^2 + T_Q,
        \end{equation*}
        for all arbitrary parameter $\gamma>0$ such that $\kappa_2 > 0$, and choices of parameters of the system such that $\kappa_0> \kappa_1\kappa_3/\kappa_2$. For example pick $\gamma = \lambda_1^2 \nu / 4\Tr Q$ so that  
        \begin{equation}
        \label{eq:condition_r}
            r   > \tfrac{2k_B}{ \nu } \left(\tfrac{h_1}{\nu}\|f\|_{-2}^2 + T_Q\right) =: r_0
        \end{equation}
    %
        \paragraph{Proof of A3.} Recall the generalized Poincar\'{e} inequality $|\Pi_n \Delta\varphi |^2 \leq \lambda_n \|\Pi_n \Delta\varphi\|^2_{-1}$. Then 
        \begin{align*}
            |G(\q, \tq)|^2  &= |a \Pi_n \Delta (\psi_1 - \tilde{\psi}_1)|^2 \leq \lambda_n a ^2\|\Pi_n\Delta (\psi_1 - \tilde{\psi}_1)\|_{-1}^2 \\
            &\leq \lambda_n a ^2\|\psi_1 - \tilde{\psi}_1\|^2 \leq \lambda_n a^2 \vertiii{\q - \tq}_{-1}^2
        \end{align*}
        giving the desired inequality with $c = a^2\lambda_n$. 
        \paragraph{Proof of A4.} We want to show that $V(\x) :=\vertiii{\x}_{-1}^2$ satisfies Assumption~\ref{ref:A4}. 
       Integrating \eqref{eq:ch2.3.1proof} over $[s,t]$ we obtain, by dropping the term $rh_1 \|\psi_2\|^2$ and estimating the forcing term as in \eqref{eq:forcing},
        \begin{multline*}
            \vertiii{\q(t)}_{-1}^2  - \vertiii{\q(s)}_{-1}^2 + \nu \int_s^t |\Delta \bfpsi|^2 \, d\tau - (t- s) T_Q \leq \\  \tfrac{ h_1 \|f\|_{-2}^2}{ \nu }(t- s) - 2h_1 \int_s^t (\psi_1, dW).
        \end{multline*}
       Rearranging and taking the expectation we have 
        \begin{equation*}
            \E\,\vertiii{\q(t)}_{-1}^2  \leq  \E\,\vertiii{\q(s)}_{-1}^2 +  \E \int_s^t (- \nu |\Delta \bfpsi|^2  + K ) \; ds 
        \end{equation*}
        where $K = \frac{ h_1|f|_{-2}^2}{\nu} + T_Q$. By Equation \eqref{eq:ch1*normpoincare} and Poincar\'{e} inequality we know that 
        \(
            \vertiii{ \q}_{-1}^2 \leq |\Delta \bfpsi|^2,
        \)
        so that
        \begin{equation*}
             \E\,\vertiii{\q(t)}_{-1}^2  \leq  \E\,\vertiii{\q(s)}_{-1}^2 +  \E \int_s^t \left(- \tfrac{\nu \lambda_1}{a_0} \vertiii{ \q(\tau)}_{-1}^2  + K \right) \; d\tau .
        \end{equation*}
        Therefore $V(\q) = \vertiii{ \q }_{-1}^2$ satisfies the estimate \eqref{eq:ch4Lyapunovfnct2} with 
        \begin{equation*}
             \gamma_1 := \frac{\nu \lambda_1}{a_0} \quand K = \frac{ h_1\|f\|_{-2}^2}{\nu} + T_Q.
        \end{equation*}
        
    \end{proof}
%
    \begin{remark}
        The existence of an invariant measure can also be proved without conditions on the parameter $r$ or any other parameter of the model. In fact it can be shown by means of the classic Krylov--Bogoliubov theorem, similarly to what was done for the 2D Navier--Stokes equations in \cite{Flandoli94}. For a complete proof of the existence of the invariant measure of the stochastic two--layer quasi--geostrophic model with this technique refer to \cite{thesis}. 
    \end{remark}
    
     \begin{remark}[Finite dimensional noise]
    From the literature (e.g.~\cite{butkovsky2020,glatt2017unique}) it is known that the coupling method applies also when the noise acts only on finitely many modes, as long as enough of them are activated. That lower bound on the dimension of the noise arose also in the argument just presented, when we required condition (ii) in \autoref{thm:exp_stab_QG}. Therefore, with few modifications to the proof of \autoref{thm:exp_stab_QG}, the ergodicity holds also for the model perturbed on the top layer only by a $n$ dimensional noise as long as \eqref{eq:condition_n} holds.
     \end{remark}

  \begin{remark}
  It is interesting to notice that the generalised coupling method used provides also a description of a potential way by which the system stabilizes.
Indeed the feedback control we introduced, namely $\Delta(\psi_1 -\tilde{\psi}_1)$, contains information only from the first layer. Then the condition on the bottom friction corresponds to a scenario in which the first layer stabilizes by the influence of the stochastic forcing and the second layer stabilizes mainly thanks to its friction. 
\end{remark}

\begin{remark}
\label{rmk:viscosity}
The result in \autoref{thm:exp_stab_QG} holds also under conditions not necessarily involving the parameter $r$ as in \autoref{thm:exp_stab_QG}. 
In fact with a simple modification in the proof of  we can retrieve a condition also, or solely, involving the viscosity. From \eqref{eq:ch3QGenergyestimate}, namely
     \begin{equation*}
        \dv{}{t}\vertiii{\bfxi}_{-1}^2 +\left(\nu - 2r \lambda_n^{-1} \right)|\Delta \bfphi|^2 \leq  \| \bfphi \|^2 \left( 2k_B |\Delta \bfpsi|^2 - 2r \right),
    \end{equation*}
    where $n$ is such that $\nu - 2a \lambda_n^{-1} > 0$, we can also use Poincar\'{e} inequality and not drop the viscosity to get
    \begin{equation*}
        \dv{}{t}\vertiii{\bfxi}_{-1}^2 + \lambda_1\left(\nu - 2r \lambda_n^{-1} \right)\| \bfphi \|^2\leq  \| \bfphi \|^2 \left( 2k_B |\Delta \bfpsi|^2 - 2r \right).
    \end{equation*}
    Using \eqref{eq:ch1*normpoincare} i.e.~$\| \bfphi\|^2 \leq \vertiii{\bfxi}_{-1}^2 \leq a_0 \|\bfphi \|^2$ 
    we derive 
    \begin{equation*}
        \dv{}{t}\vertiii{\bfxi}_{-1}^2 \leq \vertiii{\bfxi}_{-1}^2 \left( 2k_B |\Delta \bfpsi|^2 - 2r -\tfrac{\lambda_1}{a_0}\left(\nu - 2r \lambda_n^{-1} \right) \right),
    \end{equation*}
    so that, thanks to Gronwall lemma, 
    \begin{equation*}
       \vertiii{\bfxi(t)}_{-1}^2 \leq  \vertiii{\bfxi(0)}_{-1}^2\exp( - t \left( 2r+ \tfrac{\lambda_1}{a_0}(\nu - 2r \lambda_n^{-1} )\right) + 2k_B \int_0^t|\Delta \bfpsi|^2).
    \end{equation*}
    Then, Assumption~\ref{ref:A1} holds with $\kappa_0 = 2r+ \tfrac{\lambda_1}{a_0}(\nu - 2r \lambda_n^{-1} ) $.
    Since Assumption~\ref{ref:A2} holds when $\kappa_0 > \kappa_1 \kappa_3/ \kappa_2$, we require
    \begin{equation*}
        2r + \tfrac{\lambda_1}{a_0^2}\left(\nu - 2r \lambda_n^{-1} \right) >  \tfrac{2k_B}{\nu- \tfrac{2\gamma \Tr Q}{\lambda_1^2}}\left( T_Q + \tfrac{h_1}{\nu}\|f\|_{-2}^2 \right).
    \end{equation*}
    where $0<\gamma<  \lambda_1^2 \nu/ 2\Tr Q$. For example then picking $\gamma =  \lambda_1^2 \nu/ 4\Tr Q$ we have
    \begin{equation}
    \label{eq:ch3condition_r_nu}
         2r + \tfrac{\lambda_1}{a_0}\left(\nu - 2r \lambda_n^{-1} \right) >  \tfrac{4k_B}{\nu}\left( T_Q + \tfrac{h_1}{\nu}\|f\|_{-2}^2 \right).
    \end{equation}
    In particular this result provides exponential ergodicity of the model also when $r=0$ as long as the viscosity is large enough. The presence of a large viscosity would also imply that we could consider smaller values of $n$, namely more degenerate noise on the first layer. A similar result holds also for the stochastic Navier-Stokes equation. In fact in \cite{Mattingly1999} ergodicity is ensured in a large viscosity scenario even with a finite dimensional stochastic forcing. 
    \end{remark}
    Related to Remark~\ref{rmk:viscosity}, it is clear, on the one hand, that from a physical point of view we can expect ergodicity in case there is strong dissipation on both layers, for example by means of a large viscosity. On the other hand, the imposed parameter condition \eqref{eq:condition_r} requires sufficient dissipation only on one of the two layers by requiring the bottom layer (the one without noise) to be enslaved by the top one, or to converge autonomously, by means of a minimum requirement for the friction. However, a natural question which arises in this context is whether the spectral gap can be shown even when no particular condition on the dissipation is imposed. This is not clear directly from our analysis nor the available literature, nor does there exists a clear physical intuition. This will be subject of future research.


\end{document}

%% file: preamble.tex
\usepackage[english]{babel}

\usepackage{amsmath,amsthm,amsfonts,amssymb,bm}
\usepackage{mathtools}
\usepackage{hyperref} 
\usepackage{url}
\usepackage{breakcites}
\usepackage{bbm} 
\usepackage{graphicx,subfigure,caption, float}
\usepackage{array}
\usepackage{color} 
\usepackage{colortbl}
\usepackage{newlfont}
\usepackage{mathrsfs}
\usepackage{lineno} 
\usepackage[toc,page]{appendix}
\usepackage{tikz}
\usepackage{pgfplots}
\usepackage{physics}
\usepackage{fancyhdr}
\usepackage{enumitem}
\usepackage{multirow}

\usepackage{crossreftools} 
\usepackage{apptools}

\addto\extrasenglish{} 
\addto\extrasenglish{}

\usepackage{aliascnt}
\pgfplotsset{compat=1.16}   

\theoremstyle{plain}
    \newtheorem{theorem}{Theorem}[section]
    \newaliascnt{corollary}{theorem}
    \newtheorem{thmmanual}{Theorem}[section]
    \newtheorem{corollary}[corollary]{Corollary}
    \aliascntresetthe{corollary}

    \newaliascnt{lemma}{theorem}
    \newtheorem{lemma}[lemma]{Lemma}
    \aliascntresetthe{lemma}

    \newaliascnt{proposition}{theorem}
    
    \aliascntresetthe{proposition}

\theoremstyle{definition}
    \newaliascnt{definition}{theorem}
    \newtheorem{definition}[definition]{Definition}
    \aliascntresetthe{definition}

\theoremstyle{remark}
    \newaliascnt{remark}{theorem}
    \newtheorem{remark}[remark]{Remark}
    \aliascntresetthe{remark}

 \newaliascnt{example}{theorem}
    
    \aliascntresetthe{example}

\newcommand{\cP}{\mathcal{P}}
\newcommand{\cH}{\mathcal{H}}
\newcommand{\cV}{\mathcal{V}}
\newcommand{\cD}{\mathcal{D}}

\newcommand{\cB}{\mathcal{B}}

\newcommand{\cM}{\mathcal{M}}
\newcommand{\tq}{\tilde{\mathbf{q}}}
\newcommand{\tpsi}{\Tilde{\bm{\psi}}}
\newcommand{\td}{\tilde{d}}
\newcommand{\tY}{\tilde{Y}}

\newcommand{\tA}{\tilde{A}}
\newcommand{\tW}{\tilde{W}}
\newcommand{\tB}{\tilde{B}}

\newcommand{\q}{\mathbf{q}}

\newcommand{\x}{\mathbf{x}}
\newcommand{\y}{\mathbf{y}}

\newcommand{\bfphi}{\bm{\phi}}
\newcommand{\bfpsi}{\bm{\psi}}

\newcommand{\bfxi}{\bm{\xi}}
\newcommand{\R}{\mathbb{R}}
\newcommand{\N}{\mathbb{N}}
\newcommand{\E}{\mathbb{E}}
\newcommand{\bP}{\mathbb{P}}

\newcommand{\Hmo}{\mathbf{H}^{-1}} 
\newcommand{\bLtwo}{\mathbf{L}^2}

\newcommand{\bH}{\mathbf{H}}

\newcommand{\oPi}{\Pi^\perp}
\newcommand{\Wtd}[2]{W_{\tilde{d}}\left( #1, #2 \right)}
\newcommand{\fa}{\quad \text{for all }\,}
\newcommand{\quand}{\quad \text{and} \quad }

\DeclareMathOperator{\Span}{span}
\DeclareMathOperator{\range}{range}
\DeclareMathOperator{\Law}{Law}
\newcommand{\vertiii}[1]{{\left\vert\kern-0.25ex\left\vert\kern-0.25ex\left\vert #1 
    \right\vert\kern-0.25ex\right\vert\kern-0.25ex\right\vert}} 
    
\newcommand{\footremember}[2]{%
    \footnote{#2}
    \newcounter{#1}
    \setcounter{#1}{\value{footnote}}%
}
\newcommand{\footrecall}[1]{%
    \footnotemark[\value{#1}]%
} 